\documentclass[12pt,twoside,reqno]{amsart}
\usepackage{amsfonts}
\usepackage{amsmath,amscd}
\usepackage{amsthm}
\usepackage{latexsym}
\usepackage{amssymb}
\usepackage{enumerate}
\markleft{\hfill F. Diaz y Diaz\hfill E. Friedman\hfill }
\newtheorem{theorem}{Theorem}
\newtheorem{proposition}[theorem]{Proposition}
\newtheorem{corollary}[theorem]{Corollary}
\newtheorem{definition}[theorem]{Definition}
\newtheorem{lemma}[theorem]{Lemma}

\newtheorem*{claim*}{Claim}

\newcommand{\Z}{\mathbb{Z}}
\newcommand{\Q}{\mathbb{Q}}
\newcommand{\R}{\mathbb{R}}
\newcommand{\C}{\mathbb{C}}
\newcommand{\N}{\mathbb{N}}

\newcommand{\p}{\mathfrak{p}}

\newcommand{\re}{\mathrm{Re}}

\newcommand{\ie}{{\it{i.$\,$e.\ }}}
\newcommand{\Ok}{E_+}
\newcommand{\Log}{\mathrm{Log}}
\newcommand{\LOG}{\mathrm{LOG}}

\begin{document}

\title{Signed fundamental domains for totally real number  fields}

\author{Francisco Diaz y Diaz}
\email{Francisco.Diaz-y-diaz@math.u-bordeaux1.fr}
\address{Institut de Math\'ematiques de Bordeaux,  UMR 5251,
Universit\'e Bordeaux I, \newline 351, Cours de la Lib\'eration,
F-33405  Talence c\'edex, France.}

\author{Eduardo Friedman}
\email{friedman@uchile.cl}
\address{Departamento de Matem\'atica, Facultad de Ciencias,
 Universidad de Chile,\newline
Casilla 653, Santiago, Chile.}
\thanks{We are grateful for the generous support of Chilean
MIDEPLAN's Iniciativa   Cient\'{\i}fica Milenio
grant ICM P07-027-F and of  Chilean FONDECYT grant 1085153.}

\begin{abstract}
We give a  signed  fundamental  domain   for the action  on
$\R^n_+$ of the totally positive units $\Ok$ of a
totally real number field $k$ of degree $n$.
 The domain $\big\{(C_\sigma,w_\sigma)
 \big\}_\sigma$  is  signed
since the  net number of its intersections with any $E_+$-orbit is
1, \ie for any $x\in \R^n_+$,
$$
\sum_{ \sigma\in S_{n-1}}   \sum_{\varepsilon\in E_+}
w_\sigma\chi^{\phantom{1}}_{C_\sigma}(\varepsilon x )     =  1.
$$
Here   $\chi_{C_\sigma}$ is
the characteristic function of $C_\sigma$,   $w_\sigma=\pm1$
 is   a natural orientation of the  $n$-dimensional $k$-rational cone $C_\sigma\subset\R^n_+$,
and the inner sum is actually finite.

 Signed fundamental domains are as  useful as Shintani's true ones
 for the purpose of
calculating abelian $L$-functions.
They have the advantage of   being easily
constructed from any set of fundamental units, whereas in practice  there is no
algorithm producing Shintani's $k$-rational cones.

 Our proof  uses
algebraic topology   on the  quotient manifold $\R^n_+/\Ok$.
 The invariance of the topological degree under
homotopy allows us to control the deformation of a crooked
fundamental domain   into   nice straight cones.
 Crossings may occur during the homotopy,
   leading to the need to   subtract some cones.
\end{abstract}

\maketitle

\vspace*{6pt}\tableofcontents

\section{Introduction}\label{Introduction}

Explicit fundamental domains are hard to come by.
 In his 1976 work on special values of
abelian $L$-functions
 attached to a totally real number field $k$,  Shintani
 found a fundamental domain for the action of the totally positive
units $\Ok$ of $k$ on $\R_+^{[k:\Q]}$  \cite{Sh1}
 \cite[\S VII.9]{Neu}
consisting of a finite number of $k$-rational cones of varying
dimensions. Shintani's work was quite
influential but  suffered from a lack of control over the cones
involved. This differed from the
quadratic case, where a fundamental domain is easily described once
the fundamental unit is known.

For  totally real cubic fields
the situation is almost  as simple as for quadratic fields \cite{TV} (see also \cite{HP}  \cite{DF}).
In the general case, the best result is due
to Colmez \cite{Co1}\cite{Co2}. Given independent totally positive units
$\varepsilon_1,...,\varepsilon_{n-1}$, he defined $(n-1)!$ explicit  $k$-rational cones
$C_\sigma=C_\sigma(\varepsilon_1,...,\varepsilon_{n-1})$. If these units  satisfy certain geometric conditions,
 Colmez proved  that the union $\big\{ C_\sigma  \big\}_\sigma$ of his cones is
a fundamental domain for the action on $\R_+^n$ of the group
generated by the $\varepsilon_i$.\footnote{\ To be quite precise,
Colmez originally also needed somewhat less explicit lower dimensional
 cones along the boundary of the $C_\sigma$.
 Later, in unpublished lectures, he made the boundary components explicit (see \eqref{Csigma} below).}

 Colmez also proved the existence of special  units satisfying his conditions, but he gave no
 algorithm to find them, nor any upper bound on the index in $E_+$ of the
 subgroup generated by his  units. To remedy this ineffectiveness, we introduce ``signed" fundamental domains.

 When the   $\big\{ C_\sigma  \big\}_\sigma$ constitute a true fundamental domain,
 the  number of intersections of any orbit with the union of the
$C_\sigma $ is 1, \ie
$$
\sum_\sigma \sum_{\varepsilon\in E_+}\chi^{\phantom{1}}_{C_\sigma}(\varepsilon \cdot x )=1\qquad\qquad(x\in\R^n_+),
$$
where $\chi^{\phantom{1}}_{C_\sigma}$ is the characteristic function of $C_\sigma$.
In the case of a signed fundamental domain$\big\{ \big(C_\sigma ,w_\sigma\big) \big\}_\sigma$  we have
$$
\sum_\sigma w_\sigma\sum_{\varepsilon\in E_+}\chi^{\phantom{1}}_{C_\sigma}(\varepsilon  \cdot x )=1\qquad\qquad(x\in\R^n_+),
$$
where    $w_\sigma=\pm1$ is a sign assigned to each cone $C_\sigma$. In other words,
 the  net
 number  of intersections of any orbit with the $C_\sigma$  is 1.

 Using  algebraic topology we show, for {\it{any}} set of fundamental positive units, that there is a natural choice
of signs $w_\sigma=\pm1$ for which the  Colmez cones $\big\{ C_\sigma  \big\}_\sigma$ are
a signed fundamental domain. As a consequence we obtain Shintani-like formulas
for abelian $L$-functions without finding special units.

We now give a precise definition of   $w_\sigma$ and  $C_\sigma$. Here $\sigma$ runs over all
permutations of $\{1,2,\dots,n-1\}$.
  Let $\tau_i:k\to\R\,$ ($1\le i\le n)\,$
be a complete set of embeddings of   $k$, and regard $k\subset \R^n$
by identifying $x\in k$ with
$\big(x^{(1)},x^{(2)},\dots,x^{(n)}\big)\!\in\R^n$, where
$x^{(i)}=\tau_i(x)$. A unit $\varepsilon\in\Ok$ acts on
$x\in\R^n_+:=(0,\infty)^n$
 by component-wise multiplication,
$ (\varepsilon\cdot x)^{(i)}=\varepsilon^{(i)} x^{(i)}.$
We assume given independent totally positive units
$\varepsilon_1,...,\varepsilon_{n-1}$, and let $V\subset\Ok$ be the
subgroup   they   generate. To avoid  trivialities, assume $k\not=\Q$. After Colmez,  define
\begin{equation}\label{fisigma}
f_{i,\sigma} := \varepsilon_{\sigma(1)} \varepsilon_{\sigma(2)}\cdots\;
\varepsilon_{\sigma(i-1)}=\prod_{j=1}^{i-1}
 \varepsilon_{\sigma(j)}\qquad\ (1\le i\le n,
 \ \,\sigma\in S_{n-1},\ \,f_{i,\sigma}\in E_+\subset\R^n_+).
\end{equation}
For $i=1$ we mean
$f_{1,\sigma}:=1=(1,1,\dots,1)\in\R^n_+$.
Define   $w_\sigma=\pm1$ or 0 as
\begin{equation}\label{wsigma}
w_\sigma:=\frac{(-1)^{n-1}\mathrm{sgn}(\sigma)
\cdot\mathrm{sign}\big(\!\det(f_{1,\sigma},f_{2,\sigma},
\dots,f_{n,\sigma})\big)
}{\mathrm{sign}\big(\!\det(\Log\,\,\varepsilon_1,
\Log\,\,\varepsilon_2,\dots,\Log\,\,\varepsilon_{n-1})\big)},
\end{equation}
where  $\mathrm{sgn}(\sigma)$
 is the usual signature (\ie$\pm1$) of the
permutation $\sigma$,  $\, \Log\,\,\varepsilon_i\in \R^{n-1}$,
$ \big( \Log\,\,\varepsilon_i  \big)^{(j)}:= \log \, \varepsilon_i^{(j)}
\ \, (1\le j\le n-1)$,
and $\,\mathrm{sign}\big(\!\det(v_1,v_2,\dots,v_q)\big)$ is the
sign   of the determinant of the $q\times q$ matrix
 having columns $v_i$.
The determinant in the denominator of \eqref{wsigma} is the
``signed regulator" of the independent units
 $ \varepsilon_1,\varepsilon_2,\dots,\varepsilon_{n-1}$, and so non-zero.

 For  $\sigma\in S_{n-1}$ with $w_\sigma\not=0$,
  the closed cone $\overline{C}_\sigma:=
  \sum_{i=1}^n\R_{\ge0}\cdot f_{i,\sigma}\subset\R^n_+\cup\{0\}$ has a non-empty interior. Each
 bounding hyperplane
  $$
  H_{i,\sigma}:=\sum_{\substack{1\le j\le n\\ j\not=i}}\R\cdot f_{j,\sigma}\qquad\qquad \qquad \,(1\le i\le n,\ \, w_\sigma\not=0)
  $$
  separates $\R^n$ into   two  disjoint half-spaces,
  \begin{equation}\label{Hisigmaplus}\R^n=
  H_{i,\sigma}^+\cup H_{i,\sigma} \cup H_{i,\sigma}^-,
  \end{equation}
    where
  $H_{i,\sigma}^+$ is the half-space containing
  $f_{i,\sigma}$.\footnote{\ For $v\in\R^n$  we can easily compute whether
  $v\in H_{i,\sigma}^\pm$.
   On the right-hand side of
    \eqref{wsigma}   replace the single column $f_{i,\sigma}$
    by $v\in\R^n$ to obtain a function
 $v\to w_{i,\sigma}(v)$,   vanishing on
  $H_{i,\sigma}$ and    taking the    value   $ \pm w_\sigma$
    on $H_{i,\sigma}^\pm$. Alternatively,  if we write  $v=\sum_{i=1}^n c_i f_{i,\sigma}$,
    then $v\in H_{i,\sigma}^+$ if and only if $c_i>0$.\label{Hfootie}}
Fix   one of the $n$ standard basis vectors, say $e_n:=[0,0,
  \dots,0,1]\in\R^n$.
Following Colmez (unpublished lectures),
  define the   cone
 $C_\sigma$ to consist of all
    points $z\in\overline{C}_\sigma$
  for which the line segment from $e_n$
  to $z$ ``pierces"  $\overline{C}_\sigma$, \ie contains an interior point of $\overline{C}_\sigma$.
   Thus,    $C_\sigma$ consists
  of all points in the interior of $\overline{C}_\sigma$,
  together with some boundary pieces. Explicitly,
 \begin{align}\label{Csigma}
 C_\sigma= &\, C_\sigma (\varepsilon_1,\varepsilon_2,\dots,
 \varepsilon_{n-1}):=   \,\R_{1,\sigma}\cdot f_{1,\sigma}
 \,\,+\,\, \R_{2,\sigma}\cdot f_{2,\sigma}
 + \cdots \,+\,\, \R_{n,\sigma}\cdot f_{n,\sigma},
\\ \label{Risigma}
\R_{i,\sigma}=&\, \R_{i,\sigma}(\varepsilon_1,
\varepsilon_2,\dots,\varepsilon_{n-1})
 :=\begin{cases} [0,\infty) &\mathrm{ if }\ e_n\in H_{i,\sigma}^+,\\
                             (0,\infty) &\mathrm{ if }\ e_n\in H_{i,\sigma}^-,
                \end{cases}\qquad\qquad(1\le i\le n).
\end{align}
This makes sense since     $e_n$ lies in no  boundary hyperplane
  $H_{i,\sigma}$    (see Lemma \ref{envantage}).

\begin{theorem}\label{Main} Let $k$ be a totally
real number field of degree
$n\ge2$, and suppose $\varepsilon_1,\dots,\varepsilon_{n-1}$ generate a
subgroup $V$ of finite index in the group of totally positive units
  of k. Then the
signed cones $\big\{(C_\sigma, w_\sigma)\big\}_{ w_\sigma\not=0}$
defined in \eqref{wsigma} and \eqref{Csigma} give a signed fundamental domain for the
action  of $V$ on $\R^n_+:=(0,\infty)^n$. That is,
\begin{equation}\label{Basic}
\sum_{\substack{w_\sigma=+1\\ \sigma\in S_{n-1}}} \, \sum_{z\in
C_\sigma\cap V \cdot x } 1\ -\ \sum_{\substack{w_\sigma=-1\\
\sigma\in S_{n-1}}} \,
 \sum_{z\in C_\sigma\cap V \cdot x } 1 \ = \ 1\qquad\qquad\qquad\big( x\in \R^n_+\big),
\end{equation}
and all sums above are over finite sets of cardinality bounded independently of $x$.
\end{theorem}
 We prove Theorem 1 by interpreting  the left-hand side of \eqref{Basic}
     as a sum of local
  degrees   of a certain continuous map $ F:\widehat{T}\to T$ between
     a standard $(n-1)$-torus $\widehat{T}$ and the  $(n-1)$-torus $T$ coming from
      the quotient space $\R^n_+/E_+\cong T\times \R_+$. By a basic
  result in algebraic topology,  this sum of local degrees
  equals the  global degree of  $ F $.
   We compute   this global degree  to be 1  by proving that
    $ F$ is homotopic to
   an explicit    homeomorphism  $ F_0 $ of the tori involved.
      To make  the proof more accessible, we have included a short section summarizing the basics
    of topological degree theory.

   During the homotopy  from $F_0$ to $ F$ the intermediate
    maps  $ F_t$ remain  surjective,
   but not necessarily injective. Injectivity fails
   if the interior of the cones $C_\sigma$ intersect,
    leading to the need    to subtract some cones.

The condition \cite{Co1} for Colmez's special units is
$w_\sigma=+1$ for all $\sigma\in S_{n-1}$. If this holds, then
  $V\cdot x$ must intersect one and
only one of the $C_\sigma$'s. Hence we have a
 a new proof of his result.
\begin{corollary} $\mathrm{(}$Colmez  \cite{Co1}$\mathrm{)}$\label{Colmezcor}
 Suppose $w_\sigma=1$ for all $\sigma\in
S_{n-1}$, then $\displaystyle{\bigcup_{\sigma\in S_{n-1}}}C_\sigma$ is a true
fundamental domain for the action of $V$ on $\R^n_+$.
\end{corollary}
\noindent In fact, we get a slight generalization,
 as it suffices to assume $w_\sigma\not=-1$  for all $ \sigma$. Then
  $\bigcup_{w_\sigma\not=0}C_\sigma$ is still a true fundamental domain.

We now   apply  signed fundamental domains to the computation of $L$-functions.

\begin{corollary}\label{sums} Let $\mathfrak{a}_1,\dots,\mathfrak{a}_{h_+}$
be any set of integral ideals  representing   all the narrow ideal classes of
 a totally real field $k$ of degree $n\ge2$ and narrow class number $h_+$,
 let $\chi$ be  a ray-class
character of  $k$,
 and let the ideal $\mathfrak{f}$  be  the finite part of the
conductor of $\chi$. Then, for any set
 $\varepsilon_1,...,\varepsilon_{n-1}$ of generators of the
 group of totally positive units of $k$, we have
\begin{equation}\label{Lschi}
L(s,\chi)=\sum_{j=1}^{h_+} N(\mathfrak{a}_j\mathfrak{f})^{-s}
\sum_{\substack{\sigma\in S_{n-1} \\ w_\sigma\not=0
}} w_\sigma\sum_{z\in
R^\sigma\!(\mathfrak{a}_j\mathfrak{f})}\chi\big((z)\mathfrak{a}_j\mathfrak{f}\big)
\zeta^\sigma(s,z),
\end{equation}
where  $(z)$ denotes the principal fractional  ideal generated by $z\in k$,
$$
\zeta^\sigma(s,z) :=\sum_{m_1,\dots,m_n=0}^\infty\,
\prod_{j=1}^n \Big(z^{(j)}+\sum_{i=1}^n m_if_{i,\sigma}^{(j)}
\Big)^{-s}\qquad\Big(\re(s)>1,\ \,f_{i,\sigma}:=\prod_{\ell=1}^{i-1}\varepsilon_{\sigma(\ell)}\Big),
$$
is a Shintani zeta function \cite{Sh1} \cite{FR},
\begin{align}\label{Rsigmaaj}
R^\sigma (\mathfrak{a})=&\,R^\sigma (\mathfrak{a};\varepsilon_1,\dots,\varepsilon_{n-1}):=\Big\{ z\in
\mathfrak{a}^{-1}\big|\, z=\sum_{i=1}^n t_i
f_{i,\sigma},\ t_i\in I_{i,\sigma} \Big\},\\ \label{Isigmaaj}
I_{i,\sigma}:=& \,[0,1)\  \ \mathrm{if}\  \ e_n\in H_{i,\sigma}^+\ \,
\big(\mathrm{see \ \eqref{Hisigmaplus}}\big), \quad\qquad
 I_{i,\sigma}:= (0,1]\ \ \mathrm{if}\ \
e_n\in H_{i,\sigma}^-.
\end{align}
\end{corollary} \noindent Here   $\chi$ is not
necessarily primitive, it is extended by 0 to all integral ideals of
$k$  not relatively prime to $\mathfrak{f}$, and the narrow class
group Cl$_+$ is understood in its strictest sense,
 \ie an ideal $\mathfrak{a}$ represents the  trivial class in Cl$_+$
  iff $\mathfrak{a}=(z)$ for some $z\in k^*$ which  is positive at all embeddings of $k$. Note in \eqref{Rsigmaaj} that
$t_i\in \Q$  since the $f_{i,\sigma}$
are a $\Q$-basis for $k$ when $w_\sigma\not=0$. The sets
$R^\sigma(\mathfrak{a})$ are finite since
$\mathfrak{a}^{-1} \subset\R^n$ is
discrete.

Among the various expressions that Shintani gave for abelian
$L$-functions, \eqref{Lschi} closely resembles the one he published for real
quadratic fields \cite[Lemma 3]{Sh2}. In \S \ref{Corollary4}
we also give a  formula for ray class zeta functions,  analogous to \eqref{Lschi}.

We are very grateful to the referee for supplying us with an elegant proof of Lemma \ref{envantage}
below and for nudging us into simplifying our treatment of the boundaries of the cones.

\section{Signed fundamental domains}
\begin{definition}\label{vd}  A signed fundamental domain
$\{(X_i,w_i)\}_i$ for the action of a group $G$
 on a set $X$ is a finite sequence of
subsets $X_i\subset X$ and weights
$w_i\in\C$ for which there exists a constant  $K\in\R$, such that for all
$x\in X$
the cardinality $| X_i\cap G\cdot x|\le K \ \, (1\le i\le m)$,
 and
   $$\sum_{i=1}^m w_i\,| X_i\cap G \cdot x| = 1.$$\end{definition}
\noindent Note that if $Y\subset X$ is a $G$-subset,
 \ie $g\cdot y\in Y$ for all $y\in Y$ and $g\in G$,
and $\{(X_i,w_i)\}_i$ is as in Definition \ref{vd},
 then $\{(Y\cap X_i,w_i)\}_i$  is a signed fundamental domain
 for the action of $G$ on $Y$.

\begin{lemma}\label{signeddoms}
Suppose\begin{enumerate}
\item $X$ is a topological space on which the countable group $G$
acts by homeomorphisms.
\item  $ \{(X_i,w_i)\}_i$ is a signed fundamental
domain,  with each $X_i$ a Borel set $(1\le i\le m)$.
\item  $\mu$  is a positive $G$-invariant Borel
  measure (so $\mu(g\cdot A)=\mu(A)$ for any Borel set $A\subset X$
  and any $g\in G$).
  \item  $f:X\to \C$ is a Borel-measurable
   $G$-invariant function (so $f(g\cdot x)= f(x)$ for any $x\in X$ and $g\in G$).
\item The Borel set $F$ is a    true  fundamental domain  for $G$ acting on
    $X$
and   $ \int_{ F}|f(x)|\,d\mu_{(x)}<\infty$.
\end{enumerate} Then $\int_{ X_i}
|f(x)|\, d\mu_{(x)}<\infty\ \,(1\le i\le m)$  and
  $$
\int_{ F}f(x)\,d\mu_{(x)}= \sum_{i=1}^m w_i \int_{ X_i}
f(x)\, d\mu_{(x)}.
  $$
\end{lemma}
\begin{proof}
Let $\chi_i$ be the characteristic function of $X_i$.
As $F$ is a fundamental domain for the action of $G$ on $X$,
$$
\bigcup_{g\in G} (g\cdot F)=X\ \text{(countable disjoint union)},
\ \qquad \sum_{g\in G}\chi_i(g\cdot x)=| X_i\cap G\cdot x| \le K,$$
with $K$ as in the definition of a signed fundamental domain.
We have then
\begin{align*}
    \int_{ X_i}|f(x)| \,d\mu_{(x)}&=
     \int_{ X}|f(x)|\,\chi_i(x)\,d\mu_{(x)}=
      \sum_{g\in G}\int_{g\cdot F}|f(x)|\,\chi_i(x)\,d\mu_{(x)} \\
&=\sum_{g\in G}\int_{  F}|f(g\cdot x)|\,
  \chi_i(g\cdot x) \,d\mu_{(x)}=
  \sum_{g\in G}\int_{  F}|f( x)| \, \chi_i(g\cdot x) \,d\mu_{(x)}\\
&=\int_{  F}|f( x)| \Big( \sum_{g\in G}\chi_i(g\cdot x)\Big)
\,d\mu_{(x)}\le K\int_{  F}|f( x)| \,d\mu_{(x)}<\infty,
\end{align*}
proving the first claim in the lemma.
Similarly,
\begin{align*}
\int_{ F}f(x)\Big(\sum_{g\in G}\chi_i (g\cdot x)&\Big)d\mu_{(x)}
= \sum_{g\in G}\int_{ F}f(x)\chi_i(g\cdot
x)\,d\mu_{(x)}  \\
&\ = \sum_{g\in G} \int_{ F}f(g\cdot x)\chi_i(g\cdot
x)\,d\mu_{(x)} = \sum_{g\in G}\int_{g\cdot F}f(
x)\chi_i(x)\,d\mu_{(x)} \\
&\ =\int_{ X}f( x)\chi_i(x)\,d\mu_{(x)} = \int_{ X_i}f(
x) \,d\mu_{(x)}.
\end{align*}
By Definition \ref{vd}, $\sum_{i=1}^m w_i \sum_{g\in G}\chi_i(g\cdot x)=1$, so
\begin{align*}
\int_{ F} f(x)\,d\mu_{(x)}= \sum_{i=1}^m
w_i \int_{ F}f(x)\Big(
\sum_{g\in G}\chi_i(g\cdot x)\Big)d\mu_{(x)}   =\sum_{i=1}^m w_i \int_{ X_i}f(
x) \,d\mu_{(x)}.
\end{align*}
\end{proof}

\section{Proof of corollaries of main theorem }\label{Corollary4}
We first prove Corollary \ref{sums}, which we do not repeat here. Let
$\chi$ be a character of the ray class group of $k$ with conductor $\mathfrak{f}
\infty$, where $\infty$ is the formal product of all the archimedean places
of the  totally real field  $k$.
The (not necessarily primitive) $L$-function attached to $\chi$ is
$L(s,\chi):=\sum_{\mathfrak{b}}\chi(\mathfrak{b})
\mathrm{N}\mathfrak{b}^{-s}$, where $\re(s)>1,\ \,\mathfrak{b}$
ranges over all integral ideals of $k$, $\mathrm{N}$ is the absolute   norm, and $\chi(\mathfrak{b}):=0$ if
$\mathfrak{b}$ is not prime to $ \mathfrak{f}$.
Recall that we regard $k\subset \R^n$. Let $F\subset\R^n_+$
 be any true fundamental domain
 for the action of $E_+$ on $\R^n_+$.
We can pass from sums over ideals $\mathfrak{b}$
 to sums over lattice elements   $\gamma\in F$ since
for each $\mathfrak{b}$ there is a unique $j\ \, (1\le j\le h_+)$ and
$\gamma\in\mathfrak{a}_j^{-1}\mathfrak{f}^{-1}\cap F$ such that
$\mathfrak{b}=(\gamma)\mathfrak{a}_j\mathfrak{f}$.

By Theorem \ref{Main} and the remark following Definition \ref{vd},
$\big\{(C_\sigma\cap \mathfrak{a}_j^{-1}\mathfrak{f}^{-1},
 w_\sigma)\big\}_{w_\sigma\not=0}$
is a signed fundamental domain for the action of
$E_+$ on $X_j:=\mathfrak{a}_j^{-1}\mathfrak{f}^{-1}\cap \R_+^n$. Similarly, $F\cap X_j$
is a true  fundamental domain for the action of
$E_+$ on $X_j$.  Applying Lemma \ref{signeddoms}
to the discrete space $X_j$,  group $E_+$,
 counting measure $\mu$  and invariant function
    $f(\gamma):=\chi\big((\gamma)\mathfrak{a}_j\mathfrak{f}\big)
    \mathrm{N}(\gamma)^{-s}$, we find
\begin{align} 
\nonumber
L(s,\chi)&=\sum_{j=1}^{h_+}\mathrm{N}\big(\mathfrak{a}_j\mathfrak{f}\big)^{-s}
\int_{  F\cap X_j}\chi\big((\gamma)\mathfrak{a}_j\mathfrak{f} \big)
\mathrm{N}(\gamma)^{-s}d\mu_{(\gamma)}\qquad\qquad\big(\re(s)>1\big)\\
&=\sum_{j=1}^{h_+}\mathrm{N}\big(\mathfrak{a}_j\mathfrak{f}\big)^{-s}
\sum_{\substack{\sigma\in S_{n-1} \\
w_\sigma\not=0}}w_\sigma\int_{C_\sigma\cap X_j}\chi\big((\gamma)\mathfrak{a}_j\mathfrak{f} \big)
\mathrm{N}(\gamma)^{-s}d\mu_{(\gamma)}
.\nonumber
\end{align}

Thus, to prove Corollary \ref{sums} we must  show
\begin{equation}\label{need 44}
\sum_{\gamma\in C_\sigma\cap\, \mathfrak{a}_j^{-1}\mathfrak{f}^{-1}}
\chi\big((\gamma)\mathfrak{a}_j\mathfrak{f}\big)\mathrm{N}(\gamma)^{-s}=
\sum_{z\in R^\sigma\!(\mathfrak{a}_j\mathfrak{f})} \chi\big((z)\mathfrak{a}_j
\mathfrak{f}\big)\zeta^\sigma(s,z)\qquad\quad\big( \re(s)>1\big).
\end{equation}
This  was done by Shintani
 \cite{Sh2}, but we include the details here for completeness.
Recall   from \eqref{Csigma} that
$ C_\sigma := \sum_{i=1}^n  \R_{i,\sigma}\cdot f_{i,\sigma}$,
where $ f_{i,\sigma}\in E_+$ and $\R_{i,\sigma}:=[0,\infty)$
if $e_n\in H_{i,\sigma}^+$,  $\,\R_{i,\sigma}:=(0,\infty)$ if $e_n\in H_{i,\sigma}^-$.
Any $\gamma=\sum_i y_i f_{i,\sigma}\in C_\sigma$ can be uniquely written as
$\gamma=\sum_i t_i f_{i,\sigma}+\sum_i m_i f_{i,\sigma}$, where $m_i\in\Z,\ m_i\ge0,$
and $t_i\in[0,1)$ or $t_i\in(0,1]$ according to whether $e_n\in H_{i,\sigma}^+$ or not (\ie in the
notation of \eqref{Isigmaaj}, $t_i\in I_{i,\sigma}$).
 Conversely, any such $t_i$ and $m_i$ define a $\gamma\in C_\sigma$.
 Note that $\sum_i m_i f_{i,\sigma}\in \mathfrak{a}_j^{-1}
 \mathfrak{f}^{-1}$ since $f_{i,\sigma}\in
 E_+\subset\mathfrak{a}_j^{-1}\mathfrak{f}^{-1}$,
as $ \mathfrak{a}_j$ and $\mathfrak{f}$
are integral ideals. Hence
$$
z:=\sum_{i=1}^n  t_i f_{i,\sigma}\in
\mathfrak{a}_j^{-1}\mathfrak{f}^{-1}\ \Longleftrightarrow\
\gamma := \sum_{i=1}^n  y_i f_{i,\sigma}\in \mathfrak{a}_j^{-1}
\mathfrak{f}^{-1}\qquad \quad \Big(\gamma-z=\sum_{i=1}^n m_if_{i,\sigma}\Big).
$$
Hence to prove  \eqref{need 44} it suffices to prove
$ \chi\big( (\gamma) \mathfrak{a}_j\mathfrak{f} \big)=
 \chi\big( (z)\mathfrak{a}_j \mathfrak{f} \big)$.

 Note that when $\gamma\in\mathfrak{a}_j^{-1}\mathfrak{f}^{-1}$,
  the integral ideal $(\gamma) \mathfrak{a}_j \mathfrak{f} $
  is relatively prime to $ \mathfrak{f} $ if and only
   if $(z)\mathfrak{a}_j \mathfrak{f}$ is. If either ideal has a common
   factor with $ \mathfrak{f} $, we trivially have
    $\chi\big((z)\mathfrak{a}_j \mathfrak{f}\big )
 =0=\chi\big((\gamma) \mathfrak{a}_j\mathfrak{f}\big )$.
  So assume that $(z) \mathfrak{a}_j\mathfrak{f}$
  is relatively prime to $ \mathfrak{f} $. Then
  $$
 \big( (\gamma)\mathfrak{a}_j \mathfrak{f} \big)\big((z)
  \mathfrak{a}_j\mathfrak{f} \big)^{-1}=
  (\gamma z^{-1})=\Big(1+z^{-1}\sum_i m_i f_{i,\sigma}\Big).
  $$
   At primes $\p$ of $k$ dividing $\mathfrak{f} $, the valuation
 $\mathrm{ord}_\p\big((z)\mathfrak{a}_j \mathfrak{f} \big )=0$.
  Hence at such primes,
 $$
 \mathrm{ord}_\p\Big(z^{-1}\sum_i m_i f_{i,\sigma}\Big)=
 \mathrm{ord}_\p\Big( \mathfrak{a}_j\mathfrak{f}
 \sum_i m_i f_{i,\sigma}\Big)\ge
  \mathrm{ord}_\p\big(\mathfrak{a}_j \mathfrak{f}\big)
  \ge\mathrm{ord}_\p\big( \mathfrak{f}\big).
 $$
As $ 1+z^{-1}\sum_i m_i f_{i,\sigma}$ is
 totally positive,  $\chi\big( (\gamma)\mathfrak{a}_j \mathfrak{f} \big)=
 \chi\big( (z)\mathfrak{a}_j \mathfrak{f} \big)$
  by definition of the ray class group with conductor $\mathfrak{f}\infty$ \cite[p.\ 365]{Neu}.\qed

Next we prove an expression for the zeta function
$\zeta(s,\overline{\mathfrak{a}}):=\sum_{\mathfrak{b}\in
\overline{\mathfrak{a}}}\mathrm{N}\mathfrak{b}^{-s}$
attached  to a ray class $\overline{\mathfrak{a}}$  modulo
 $\mathfrak{f}\infty$.  Here $\mathfrak{b}$
  runs over all integral ideals in
 $\overline{\mathfrak{a}}$, and the ray classes are
again taken in the strictest sense, \ie
  $\mathfrak{f}\infty$  is  the formal
product of   an integral ideal  $\mathfrak{f}$ with
all $n$ archimedean places of the totally real field $k$.
\begin{corollary}\label{partialzeta}
Suppose    $\eta_1,\dots,\eta_{n-1}$ generate the
group $E_\mathfrak{f}^+$ of
totally positive units of  $k$ which are congruent
 to $\mathrm{1}$ modulo $\mathfrak{f}$,
let   $\mathfrak{a}\in\overline{\mathfrak{a}}$ be
 an integral ideal and
 let   $\Z\cap\mathfrak{f}=:f\Z$, with $ f\in\N$. Then
$$
\zeta(s,\overline{\mathfrak{a}})=\mathrm{N}\mathfrak{a}^{-s}
\sum_{\substack{\sigma\in S_{n-1}\\ w_\sigma\not=0
}} w_\sigma\sum_{z\in
R_{\mathfrak{f},\mathfrak{a}}^\sigma}\zeta_\mathfrak{f}^\sigma(s,z)\qquad\qquad\big(\re(s)>1\big),
$$
where
\begin{align*}
\zeta_\mathfrak{f}^\sigma(s,z)&:=\sum_{m_1,\dots,m_n=0}^\infty\,
\prod_{j=1}^n \Big(z^{(j)}+f\sum_{i=1}^n m_ig_{i,\sigma}^{(j)}
\Big)^{-s}\qquad \qquad\qquad \Big(g_{i,\sigma} :=\prod_{\ell=1}^{i-1}\eta_{\sigma(\ell)}\Big),
\\
R_{\mathfrak{f},\mathfrak{a}}^\sigma &:=\Big\{ z\in
1+\mathfrak{a}^{-1}\mathfrak{f}\big|\, z=f\sum_{i=1}^n t_i g_{i,\sigma},
\ t_i\in I_{i,\sigma} \Big\},
\\
I_{i,\sigma}&:= \begin{cases}[0,1)\  & \mathrm{if}\ r_i>0
 \ \mathrm{when \ we \ write}\   e_n=[0, \dots,0,1] =\sum_{j=1}^n r_jg_{j,\sigma} ,\\
  (0,1]\ & \mathrm{otherwise.}
\end{cases}
\end{align*}
\end{corollary}
\begin{proof}
Using a fundamental domain $F_\mathfrak{f}$ for the action
 of $E_\mathfrak{f}^+$ on $\R^n_+$, we
 re-write the sum over $\mathfrak{b}$
 defining $\zeta(s,\overline{\mathfrak{a}})$,  letting
 $\mathfrak{b}=\mathfrak{a}(\gamma)$, where
 $ \gamma \in  1+\mathfrak{a}^{-1}\mathfrak{f}$
  and $\gamma \in  F_\mathfrak{f}$.
 From here on we proceed as in the
 proof of Corollary \ref{sums}, replacing
 $\mathfrak{a}_j^{-1}\mathfrak{f}^{-1}$ by
  $1+\mathfrak{a}^{-1}\mathfrak{f}$, $\ \ \ \ \quad F$
   by $F_\mathfrak{f}$, and $E_+$ by $E_\mathfrak{f}^+$. The definition of
   $I_{i,\sigma}$ in Corollary \ref{partialzeta}
    differs formally from the one given in \eqref{Isigmaaj} because this time
   we used footnote \ref{Hfootie} to describe the hyperplanes determined by the faces of the cone $C_\sigma(\eta_1,\dots,\eta_{n-1})$.
 In the proof of Corollary \ref{partialzeta} we need not
  worry about character values, but we must use generators
 of $C_\sigma$  in $\mathfrak{a}^{-1}\mathfrak{f}$, hence the  need for the $fg_{i,\sigma}$.
\end{proof}

\section{From cones to polytopes}\label{lowering}
Since  we are interested only
in cone domains, signed or not,
it is natural to consider the action of $V$ on the
set $\mathcal{L}$ of  half-lines in $\R^n_+$ emanating from 0.
The action by $\varepsilon\in V$ takes half-lines to half-lines, so one
easily sees that a
  fundamental domain for the action of $V$ on $\mathcal{L}$
automatically yields a cone fundamental domain
 for the action of $V$ on $\R^n_+$, and
conversely. In   this section
 we extend this old idea to signed fundamental domains.

For $n\ge2$ and $x\in\R^n$ with non-vanishing last coordinate $x^{(n)}$,
  define $\ell(x)\in\R^{n-1}$ as
\begin{equation}\label{ell}
 \ell (x):=\Big( \frac{x^{(1)}}{x^{(n)}},
 \frac{x^{(2)}}{x^{(n)}},\dots, \frac{x^{(n-1)}}{x^{(n)}}
 \Big)  \qquad\qquad\big(x\in\R^n,\ \, x^{(n)}\not=0\big).
\end{equation}
The  reason for the usefulness of $\ell$ is that the intersection of
the half-line $L_x:=\{tx\}_{t\in\R_+}$ with the hyperplane $x^{(n)}=1$
occurs at the point $\big(\ell(x),1\big)$. For any $y\in\R^{n-1}_+$,
    the set of $x\in\R^{n}_+$ satisfying $\ell(x)=y$  is exactly the half-line
    $L_{(y,1)}$.

  Define
\begin{equation}\label{vtilde}
  {\widetilde{V}}:=\ell(V)=\langle{\widetilde{\varepsilon}}_1,\dots,
  {\widetilde{\varepsilon}}_{n-1}\rangle\subset\R_+^{n-1},\qquad
   {\widetilde{\varepsilon}}_i:=\ell(\varepsilon_i),
\end{equation}
 where
$V:=\langle\varepsilon_1,\dots,\varepsilon_{n-1}
\rangle\subset\Ok\subset\R^n_+ ,$ as in Theorem \ref{Main}.
We regard Euclidean space as a ring  under coordinate-wise
  multiplication, so ${\widetilde{V}}$ acts on $\R_+^{n-1}$. 
The next result  will let us pass from   $(n-1)$-simplices to
  $n$-cones in the proof of Theorem \ref{Main}.
\begin{lemma}\label{loseone}  If
$\{(\gamma_i,w_i)\}_i$ is a signed fundamental domain for the
action of ${\widetilde{V}}$ on $\R^{n-1}_+$,  then $\{(\Gamma_i,w_i)\}_i$ is
a signed fundamental domain for the action of $V$ on $\R^n_+$,
where
$\Gamma_i:=\big\{x\in\R^n_+\big|\,\ell(x)\in\gamma_i\big\}$.
\end{lemma}
\begin{proof}
For $x\in \R^n_+$, let us prove that $\ell$ induces a bijection between
$\Gamma_i\cap V\cdot x$ and $\gamma_i\cap
{\widetilde{V}}\cdot  \ell(x)$.
 Indeed, since
 $\ell(\varepsilon\cdot x)= \ell(\varepsilon)\cdot \ell(x)$, it is clear that $\ell$ maps
$\Gamma_i\cap V\cdot x$ surjectively onto  $\gamma_i\cap
{\widetilde{V}}\cdot
 \ell(x)$.
If  $\ell(\varepsilon\cdot x)=\ell(\varepsilon^\prime\cdot x)$
 for $\varepsilon,\varepsilon^\prime\in V$,
  then $\ell(\varepsilon^{-1}\varepsilon^\prime)
  =1_{n-1}:=(1,1,\dots,1)\in\R_+^{n-1}$.
  But $\ell\big((\delta^{(1)},\delta^{(2)},\dots,\delta^{(n)})\big)=1_{n-1}$ implies
 $ \delta^{(1)}=\delta^{(2)}=\cdots=\delta^{(n)}$.
 For $\delta\in\Ok$, this means $\delta=1$, as $\prod_{i=1}^n \delta^{(i)}=1.$ Hence $\ell$ is injective
 on $V\cdot x$ (for $x$ fixed). The lemma now follows directly from Definition \ref{vd} of
 a signed fundamental domain.\end{proof}

We shall apply the next lemma  to relate  a
cone in $\R^n_+$ with the polytope resulting from its
  intersection   with the hyperplane $x^{(n)}=1$.
\begin{lemma}\label{Projection} Suppose
$x,T_0,T_1,\dots,T_h\in\R^n$ all have non-zero last coordinate. If
$x=\sum_{i=0}^h c_i T_i$ with $ c_i\in\R$,
then
$ \ell(x)=\sum_{i=0}^h b_i\ell(T_i),$ where $ b_i =
 c_iT_i^{(n)}/x^{(n)},$  and $\sum_{i=0}^h b_i=1.$
Conversely, if
 $\ell(x)=\sum_{i=0}^h b_i\ell(T_i),$ where $\sum_{i=0}^h b_i=1$ and $b_i\in\R$,
 then  $  x=\sum_{i=0}^h c_i T_i,$ where  $ c_i=x^{(n)}b_i/T_i^{(n)}.$ In particular,
 if  $T_i^{(n)}>0\ \,(0\le i\le h)$ and $x^{(n)}>0$, then $c_i>0$ if and only if $b_i>0$.
\end{lemma}

\begin{proof}
For $T\in\R^n$ with $T^{(n)}\not=0$, definition \eqref{ell}
of $\ell$ gives the obvious identity   $$
T=\big(T^{(1)},\dots,T^{(n)}\big)=
T^{(n)} \big(\ell(T),1\big).
$$
If $x= \sum_{i=0}^h c_i T_i$, then $ x^{(j)}=\sum_{i=0}^h c_i T_i^{(j)}$. Hence
\begin{align*}
\ell(x)=& \frac{1}{x^{(n)}} \bigg(\sum_{i=0}^h c_i T_i^{(1)},\sum_{i=0}^h c_i
T_i^{(2)},\dots,\sum_{i=0}^h c_i T_i^{(n-1)}
\bigg)
\\
  =&   \sum_{i=0}^h \frac{c_i T_i^{(n)}}{x^{(n)}}  \Big(
\frac{T_i^{(1)}}{T_i^{(n)}},\frac{T_i^{(2)}}{T_i^{(n)}},\dots,
\frac{T_i^{(n-1)}}{T_i^{(n)}} \Big)= \sum_{i=0}^h
\frac{c_iT_i^{(n)}}{x^{(n)}} \ell(T_i)=\sum_{i=0}^h b_i\ell(T_i).
\end{align*}
As $x^{(n)}=\sum_{i=0}^h c_i T_i^{(n)}$,
we have $\sum_i b_i=\sum_i  (c_iT_i^{(n)}/x^{(n)})=1$.

Conversely, if $\ell(x)=\sum_{i=0}^h b_i\ell(T_i)$ with   $\sum_{i=0}^h b_i=1$, then
\begin{align*}
x &=   x^{(n)} \big(\ell(x),1\big)=x^{(n)} \big(\sum_{i=0}^h
b_i\ell(T_i),1\big)=  x^{(n)}
\big(\sum_{i=0}^h b_i\ell(T_i),\sum_{i=0}^h b_i\big)\\
  &=\,  \sum_{i=0}^h x^{(n)}b_i \big(\ell(T_i),1\big)=
  \sum_{i=0}^h \frac{x^{(n)}b_i}{T_i^{(n)}}\,T_i^{(n)}
  \big(\ell(T_i),1\big)=\sum_{i=0}^h \frac{x^{(n)}b_i}{T_i^{(n)}}\,T_i=
  \sum_{i=0}^h c_i T_i.
\end{align*}
\end{proof}

The next lemma, on taking  $Q=\Q,\, R=\R$ and   $k$ a totally
 real number field,
  shows that a standard basis vector cannot line up with any
 face of a   $k$-rational cone, \ie $e_n\notin H_{i,\sigma}$
  for $1\le i\le n$ and $\sigma\in S_{n-1}$  as claimed after
  \eqref{Risigma}.

\begin{lemma}\label{envantage} Let $Q\subset  k\subset R$
be a tower of fields, with    $k/Q$ a finite separable extension.
 Let $v_1,v_2,\dots,v_\ell\in k$ with $\ell<n:=[k:Q]$, let
  $\tau_i:k\to R$ be the $n$ distinct field homomorphisms
  of $k$ into $R$ fixing $Q\ \,(1\le i\le n)$,  and
   define $J:k\to R^n$  by $\big(J(v)\big)^{(i)}:=\tau_i(v)$ for $v\in k$.
    Then $e_n:=[0,0,\dots,0,1]\in R^n$ is not contained in
     the $R$-subspace $R\cdot J(v_1) +R\cdot J(v_2)
    +\dots+R\cdot J(v_\ell)  \subset R^n$.
\end{lemma}

\begin{proof}
Since $k/Q$ is separable and $\ell<n$, there exists a nonzero $x\in k$ such that
Tr$_{k/Q}(xv_i)=0$ for $i=1,...,\ell$. Let $\psi:R^n\to R$ be the $R$-linear map given by
dot product with $J(x)$. Then $\psi(v_i)=\mathrm{Tr}_{k/Q}(xv_i)=0$,  whereas $\psi(e_n)=\tau_n(x)\not=0$.
Thus $e_n$ is not in the $R$-span of the $v_i$.
\end{proof}

We can now describe the simplices $c_\sigma$ that result from intersecting the cones
$C_\sigma$  with the hyperplane $x^{(n)}=1$.
 Let
\begin{align}\label{csigma}
 c_\sigma&:=   \Big\{y\in\R^{n-1}_+\big| y=
 \sum_{i=0}^{n-1}b_i\varphi_{i,\sigma},\ \,\sum_{i=0}^{n-1}b_i=1,
 \ \,b_i\in J_{i,\sigma}  \Big\}\quad\quad(\sigma\in S_{n-1},\ \,w_\sigma\not=0),\\
  \nonumber
& \varphi_{i,\sigma}:=\ell(f_{i+1,\sigma}),\quad \quad
J_{i,\sigma}:=
\begin{cases} [0,1]   &\mathrm{ if\ }  e_n\in H_{i+1,\sigma}^+,\\
 (0,1] &\mathrm{ if\ }  e_n\in H_{i+1,\sigma}^-,
 \end{cases}\qquad \quad \ (0\le i\le n-1).
 \end{align}
Note the annoying index shift between \eqref{Csigma} and \eqref{csigma},
 $\varphi_{i,\sigma}:=\ell(f_{i+1,\sigma})$.

The next result  restates Theorem \ref{Main} in
terms of the $c_\sigma$.
\begin{proposition}\label{conestosimplices} If
  $\big\{(c_\sigma,w_\sigma)\big\}_{w_\sigma\not=0}$
is a signed fundamental domain for the action of
$\widetilde{V}$ on $\R^{n-1}_+$
$\big($see \eqref{vtilde}$\big)$, then
 $\big\{(C_\sigma,w_\sigma)\big\}_{w_\sigma\not=0}$
is a signed fundamental domain for the action of $V$ on $\R^n_+$.
\end{proposition}
\begin{proof} Lemma \ref{loseone} shows that we must only prove $
C_\sigma=\big\{x\in\R^n_+\big|\,\ell(x)\in c_\sigma \big\}.$
So suppose $x\in C_\sigma$. Then
  $x= \sum_{i=1}^n c_i f_{i,\sigma}$,
where $c_i\ge0$ if $e_n\in H_{i,\sigma}^+,$
but $c_i>0$ if $ e_n\in H_{i,\sigma}^-$ $\big($see \eqref{Csigma} and \eqref{Risigma}$\big)$.
Note $f_{i,\sigma}^{(n)}>0\ \,(1\le i\le n)$ and $x^{(n)}>0$.
Lemma \ref{Projection} shows
$$
\ell(x) =\sum_{i=0}^{n-1} b_i\varphi_{i,\sigma},
\qquad\sum_{i=0}^{n-1} b_i=1,\qquad b_i =
 c_{i+1}f_{i+1,\sigma}^{(n)}/x^{(n)}\ge0\qquad(0\le i\le n-1),
$$
from which   it is clear that $b_i\le1$. Since $b_i=0$ is possible
 only if $c_{i+1}=0$, \ie $e_n\in H_{i+1,\sigma}^+$, we have $\ell(x)\in c_\sigma$.
 Thus, $C_\sigma\subset\big\{x\in\R^n_+\big|\,\ell(x)\in c_\sigma \big\}.$

To prove the reverse inclusion, suppose $x\in\R^n_+$ and
 $\ell(x)=\sum_{i=0}^{n-1}b_i\varphi_{i,\sigma}\in c_\sigma$.
Lemma  \ref{Projection} and \eqref{csigma} show
that $x=\sum_{i=1}^n c_i f_{i,\sigma}$,
with $c_{i}=b_{i-1}x^{(n)}/f_{i,\sigma}^{(n)}\ \,(1\le i\le n)$.
 Thus $c_i\ge0$, with equality possible only if
 $e_n\in H_{i,\sigma}^+$. Hence $x\in C_\sigma$, as claimed.
\end{proof}

\section{The piecewise affine map}\label{Affs}
In the previous section we reduced the proof of Theorem \ref{Main} to proving
that the  simplices $c_\sigma$ give a signed fundamental domain. After some
affine preliminaries, in this section we
interpret   $\bigcup_{\sigma\in S_{n-1}} \overline{c}_\sigma$  as the image $f([0,1]^{n-1})$ of a
 hypercube
 by a (continuous) piecewise
affine map. Each $\overline{c}_\sigma=f(D_\sigma)$
for a simplex $D_\sigma\subset [0,1]^{n-1}$.
Then we show that the difference between $c_\sigma$ and its closure $\overline{c}_\sigma$
 can be interpreted in terms of ``simplex piercing."

\subsection{Polytopes and affine maps}\label{PAF}
If $w_0,\dots,w_{r}$ are  elements of a real vector space $W$, the
 (closed) polytope they generate is the set   of convex sums
\begin{equation}\label{polytope}
P=P(w_0,\dots,w_r):=\Big\{ w\in W\big|\,w=\sum_{i=0}^{r} b_iw_i,\ \
b_i\ge0,\ \ \sum_{i=0}^{r}b_i=1,  \Big\}.
 \end{equation}
In general, if $w\in W$ and
\begin{equation}\label{barycoords}
w=\sum_{i=0}^{r} b_i w_i,\qquad \ b_i\in\R,  \qquad
\sum_{i=0}^{r}b_i=1 ,
 \end{equation}
 the $b_i$ are called barycentric coordinates of $w$ with
respect to $w_0,\dots,w_{r}$. They are uniquely determined  if and
only if the $r$
vectors $\big\{w_i-w_j\big\}_{\substack{0\le i \le r\vspace{.05cm}\\
i\not=j}}$ are $\R$-linearly independent   (for any fixed
 index $j\in\{0,1,\dots,r\}$). Then we call
$w_0,\dots,w_{r}$ affinely     independent and
$P=P(w_0,\dots,w_{r})$ an $r$-simplex with vertices $w_i$.
Vertices  are  uniquely determined
(up to re-ordering)  by
the $r$-simplex $P\subset W$.\footnote{\ Proof: The vertices $w_i$ are the only elements
$w\in P$ which cannot be written as $w=tv_1+(1-t)v_2$ with
$v_1,v_2\in P,\ v_1\not=v_2,\ 0<t<1$.} If $\dim(W)=r$ and the $r+1$ vertices of $W$ are affinely independent,
 we call them an affine basis
of $W$. In this case we   write $b_i(w) $
for  the $b_i$ in \eqref{barycoords}. Barycentric coordinates satisfy
\begin{equation}\label{baryconvex}
b_i\big((1-t)x+ty\big)=(1-t)b_i(x)+tb_i(y)\qquad\qquad(t\in\R,\ \,x,y\in
W,\ \,0\le i\le r).
\end{equation}

A  face   of a polytope $P=P(w_0,\dots,w_{r})$ for us is a subset
\begin{equation}\nonumber 
P_j:= 
\Big\{ w\in W\big|\,w=\sum_{\substack{0\le i\le r\vspace{.05cm}\\ i\not=j}} b_iw_i,\ \
b_i\ge0,\ \ \sum_{\substack{0\le i\le r\vspace{.05cm}\\ i\not=j}} b_i=1,  \Big\}.
 \end{equation}
 The affine subspace $h_j$ containing $P_j$ is
 \begin{equation}\label{affinesubspaces}
h_j:= 
\Big\{ w\in W\big|\,w=
\sum_{\substack{0\le i\le r\vspace{.05cm}\\ i\not=j}} b_iw_i,\ \
b_i\in\R,\ \ \sum_{\substack{0\le i\le r\vspace{.05cm}\\ i\not=j}} b_i=1,  \Big\}.
 \end{equation}

An affine map $A:W\to W^\prime$
between  real  vector spaces  has the form  $A(w)= q+L(w)$ for
 a unique   $q=A(0)\in W^\prime$ and a unique
linear map $L:W\to W^\prime$, called the linear part of $A$. If $w_0,\dots,w_{r}$ is an
affine basis of $W$ and $p_0,\dots,p_{r}$ are arbitrary elements of
$W^\prime$, there is   a unique affine map $A:W\to W^\prime$ such
that $A(w_i)=p_i$ for $0\le i\le r $. Indeed,   let $L$
  be the unique linear map such that
$L(w_i-w_0)=p_i-p_0$ for   $1\le i\le r$, and set $q=p_0-L(w_0)$.
Then $A(w)=q+L(w)$ is the required affine map. Its uniqueness is
clear.

If $w\in W$ has barycentric coordinates $b_i \ (0\le i\le r)$ with
respect to $w_0,\dots,w_{r}$, and $A:W\to W^\prime$ is an affine map
 with $A(w_i)=p_i\ (0\le i\le r)$, then the same $b_i$ are also
 barycentric coordinates   for $A(w)$ with respect to
$p_0,\dots,p_{r}$. They are the unique such coordinates if and only
if the $p_i$ are affinely independent, \ie if and only if the
associated linear map $L$ is injective.
We record this as
\begin{equation}\label{bary}
A(w_i)=p_i\ (0\le i\le r),\ w=\sum_{i=0}^{r} b_i(w)w_i,\
 \sum_{i=0}^{r}b_i(w)=1\ \ \Longrightarrow\ \ A(w)=\sum_{i=0}^{r}b_i(w) p_i,
\end{equation}
valid whenever the $w_i$ are an affine basis of $W$.
An affine map  $A:W\to W^\prime$ is bijective if and only if it takes an affine
basis of $W$ to an affine basis of $W^\prime$.

\subsection{The Colmez piecewise affine map}\label{cones}

 Let $C=\bigcup_i Q_i\subset
W$ be a finite union of    polytopes $Q_i$ inside a real vector space $W$.
If $W^\prime$ is also such a
  space, we will call a map $f:C\to W^\prime$  piecewise affine if $f$
restricted to each $Q_i$ is the restriction to $Q_i$
of an affine map $A_i:W\to W^\prime$.  Then, of course,
$A_i(x)=A_j(x)=f(x)$ for $x\in Q_i\cap Q_j$. Conversely,
given polytopes $Q_i\subset W$ and affine maps $A_i:W\to W^\prime$
with $A_i(x)=A_j(x)$ for $x\in Q_i\cap Q_j$, there is a
unique piecewise affine map $f:\bigcup_i Q_i\to W^\prime$
restricting to $A_i$ on each $Q_i$. We note that a piecewise affine map is
necessarily continuous.

We decompose the unit $(n-1)$-cube into  $(n-1)!$ simplices according to the order of
the coordinates, \ie
\begin{equation}\label{Cubedecomp}
I^{n-1}:=[0,1]^{n-1}=\bigcup_{\sigma\in S_{n-1}} D_\sigma,
 \end{equation}
 where for each permutation $\sigma$ of $\{1,\dots,n-1\}$ we set
\begin{equation}\label{Dsigma}
 D_\sigma:=\Big\{ x=\big(x^{(1)},\dots,x^{(n-1)}\big)
 \!\in I^{n-1} \big|\,  x^{(\sigma(1))}\ge x^{(\sigma(2))}
 \ge \cdots\ge x^{(\sigma(n-1))}\Big\}  .
 \end{equation}
Let   $e_i\in\R^{n-1} \ (1\le i\le n-1)$ be the
$i^{\mathrm{th}}$ standard basis vector,   so $e_i$ has a 1 in the
$i^{\mathrm{th}}$ coordinate and zeroes elsewhere. One checks that the $n$ vertices of
$D_\sigma$ are
\begin{equation}\label{phiisigma}
 \phi_{i,\sigma}:= \sum_{j=1}^{i}e_{\sigma(j)} \qquad\qquad\qquad (0\le i\le n-1,\ \,\phi_{0,\sigma}:=0),
\end{equation}
and that they are affinely independent.

 We return to the context of Theorem \ref{Main}.   Thus
$V=\langle\varepsilon_1,\dots,\varepsilon_{n-1}\rangle\subset\Ok$ is
a subgroup of finite index in the group of totally positive units of a
totally real field $k$ of degree $n$, thought of as embedded in $\R^n$.
Recall that we   defined in \eqref{ell} a map $\ell:\R^n-\big\{x^{(n)}=0\big\}\to\R^{n-1}$,
and that $\widetilde{V}:=\ell(V)\subset\R^{n-1}_+$ acts on $\R^{n-1}_+$ by component-wise
multiplication.

For $\sigma\in S_{n-1}$,  define  $A_\sigma:\R^{n-1}\to\R^{n-1}$ to
be the unique affine map such that
 \begin{equation}\label{Asigma}
A_\sigma(\phi_{i,\sigma}):=\varphi_{i,\sigma}
 \qquad \qquad(0\le i\le n-1),
\end{equation}
where   $\varphi_{i,\sigma}:=\ell(f_{i+1,\sigma})\in\R^{n-1}_+$, as in \eqref{csigma}.
There we only dealt with $\sigma$ such that $w_\sigma\not=0$, but here we will need to deal with
all $\sigma\in S_{n-1}$.

The next proposition shows that the  $A_\sigma$ can be glued together to get a piecewise
affine map $f$ on the unit hypercube.

\newpage

\begin{proposition}\label{f} There is a continuous map
$f:I^{n-1}\to\R^{n-1}_+$ with the following properties.
\begin{itemize}
\item[(i)] If $x\in D_\sigma$, then $f(x)=A_\sigma(x)$, the affine map defined in \eqref{Asigma}.
\item[(ii)] If $x\in I^{n-1}$ and $x+e_i\in I^{n-1}$  for some element
$e_i$ of the standard basis of $\R^{n-1}$, then
 $ f(x+e_i) = \widetilde{\varepsilon}_i\cdot f(x),$
 where $\widetilde{\varepsilon}_i:=\ell(\varepsilon_i)\ \,(1\le i\le n-1)$.
\item[(iii)] If $x=\sum_{i=1}^{n-1}b_i e_i$ is a vertex
of the cube $I^{n-1}$, then
$f(x)=\prod_{i=1}^{n-1}{\widetilde{\varepsilon}}_i^{\,b_i}$.
\end{itemize}
\end{proposition}
\noindent Note that in (iii), $b_i=1$ or 0, and
 $\widetilde{\varepsilon}_i^{\,0}:=1=1_{n-1},$
 the identity of the ring $\R^{n-1}$.
\begin{proof}
Since $I^{n-1}=\bigcup_\sigma D_\sigma$, to prove the existence of a continuous $f$
satisfying (i) we
need to show that if $x\in D_\sigma \cap D_\tau$ for
$\sigma\not=\tau\in S_{n-1}$, then $A_\sigma(x)=A_\tau(x)$.
A  vertex $v=(v^{(1)},\dots,v^{(n-1)})=\phi_{i,\sigma}\in D_\sigma$
satisfies
\begin{equation}\label{vvertices1}
v^{(\sigma(j))}=\begin{cases}
                                1&\ \text{if }j\le i,\\
                                0&\ \text{otherwise.}
              \end{cases}
\end{equation}
In other words,  for $1\le m\le n-1$, we have $v^{(m)}=1$ if $m=\sigma(j)$
for some $j\le i$, but $v^{(m)}=0$ otherwise.
Hence
\begin{equation}\nonumber
A_\sigma(v)  = A_\sigma(\phi_{i,\sigma}):=\ell(f_{i+1,\sigma}):=
\ell\Big(\prod_{j=1}^i  \varepsilon_{\sigma(j)}\Big)=
\prod_{j=1}^i\ell(\varepsilon_{\sigma(j)}) =\prod_{j=1}^i\widetilde{\varepsilon}_{\sigma(j)}
=\prod_{m=1}^{n-1}\widetilde{\varepsilon}_{m}^{\,v^{(m)}}.
\end{equation}
As this last expression is independent of $\sigma$, we have $ A_\sigma(v) = A_\tau(v)$
if $v$ is a vertex of $D_\sigma$ and of $D_\tau$.
But  $P_{\sigma,\tau}:=D_\sigma\cap D_\tau$ is a $d$-simplex (for some $ 1\le d\le n-2$)
 whose $d+1$  vertices
 are also vertices of $D_\sigma$ and of $D_\tau$.  An affine map on
 a  $d$-simplex is uniquely determined
 by its values on  the $d+1$ vertices, so $A_\sigma(x)=A_\tau(x)$
 for all $x\in  P_{\sigma,\tau}:=D_\sigma\cap D_\tau$, proving (i).

To prove (ii),   suppose $x\in I^{n-1}$ and $x+e_i\in I^{n-1}$ for some $i$.
This implies $x^{(i)}=0$, so  $x\in D_\sigma$ for some
$\sigma\in S_{n-1}$ such that $\sigma({n-1})=i$
$\big($see \eqref{Dsigma}$\big)$. Write
$x=\sum_{j=0}^{n-1}b_j\phi_{j,\sigma}$ in the barycentric coordinates associated to
 $D_\sigma$, so $b_j\ge0$ and $\sum_{j=0}^{n-1}b_j=1$. Then
$b_{n-1}=0,$  for otherwise $x^{(i)}=x^{(\sigma(n-1))}>0$.
Notice that $x+e_i\in D_{\tilde{\sigma}}$,
where $\tilde{\sigma}\in S_{n-1}$ is given by
$$
\tilde{\sigma}(1)=i,\qquad\qquad
\tilde{\sigma}(j)=\sigma(j-1)\qquad(2\le j\le{n-1}).
$$
Hence,
\begin{equation}\label{otherside}
\phi_{j,\tilde{\sigma}}= e_i +
\phi_{j-1,\sigma},\qquad\qquad\varphi_{j,\tilde{\sigma}}
={\widetilde{\varepsilon}}_i
\varphi_{j-1,\sigma}\qquad (1\le j\le {n-1}).
\end{equation}
From this one checks that  the barycentric coordinates associated to
 $D_{\tilde{\sigma}}$ giving  $x+e_i=\sum_{j=0}^{n-1}\tilde{b}_j\phi_{j,\tilde{\sigma}}$
 are
$$
\tilde{b}_0=\,0,\qquad\tilde{b}_j=  \,b_{j-1} \qquad\qquad(1\le j\le {n-1}).
$$
By (i), we may use $A_{\tilde{\sigma}}$ to calculate
$f(x+e_i)$ and $A_\sigma $ for $f(x)$. From
\eqref{bary} and \eqref{otherside},
\begin{align*}
f(x+e_i)&=A_{\tilde{\sigma}}(x+e_i)=\sum_{j=0}^{n-1}
\tilde{b}_j\varphi_{j,\tilde{\sigma}}=\sum_{j=1}^{n-1}
\tilde{b}_j\varphi_{j,\tilde{\sigma}} =\sum_{j=1}^{n-1} b_{j-1}{\widetilde{\varepsilon}}_i\varphi_{j-1, \sigma
}\\
& ={\widetilde{\varepsilon}}_i\sum_{j=0}^{n-2} b_j\varphi_{j, \sigma }=
{\widetilde{\varepsilon}}_i\sum_{j=0}^{{n-1}} b_j\varphi_{j, \sigma
}={\widetilde{\varepsilon}}_i A_\sigma(x)={\widetilde{\varepsilon}}_i f(x),
\end{align*}
proving (ii).

Since $f(0)=f(\phi_{0,\sigma})=\varphi_{0,\sigma}=\ell(f_{1,\sigma})=\ell(1)=1$,
claim (iii) follows from (ii) by induction on the number of non-zero coordinates of the vertex.
\end{proof}

\subsection{Piercing}\label{Piercing}
We now make an {\it{ad hoc}} definition,    which we will later use to study the boundary of the
signed fundamental domain in Theorem \ref{Main}.
\begin{definition} \label{defpiercing}  Suppose  $ P\subset W$ is a subset of
 some finite-dimensional real
vector space $W$. For  $ x,y\in  W,$ we shall say that $\overrightarrow{x ,y} $
pierces $P$  if    $\,y\in P$
  and the closed line segment  $\overrightarrow{x ,y} $ connecting $x$
  and $y$ intersects the interior $\stackrel{\circ}{P}$ of $P$.
\end{definition}
\noindent
Note the asymmetry between the initial point $x$ and the final point $y$ in the above definition. The final
point must be in $P$, but the initial point need not be. If $x=y$, piercing
is equivalent to $y\in \ \stackrel{\circ}{P}$. In general, there obviously is
piercing   if  either $x$  or $y$ lie in $\stackrel{\circ} {P}$.
Of course, piercing cannot occur if $P$ has an empty interior.

A practical way of determining piercing for
 an $r$-simplex
is through the barycentric
coordinates $b_i(x)$ and $b_i(y)$.
\begin{lemma} \label{piercing} Let   $W$ be a real vector space
of dimension $r$, let $x\in W$ and let $y\in P=P(w_0,w_1,\dots, w_r)$,
 an $r$-simplex  in  $W$.
Then $\overrightarrow{x ,y} $ pierces $P$ if and
only if  $ b_i(x)>0$ whenever  $ b_i(y)=0\ \,(0\le i\le r).$  Moreover, if
 $ z$ lies in the interior $ \stackrel{\circ} {P} $ of $P$, then so do
all points of $\,\overrightarrow{ z ,y}$, except possibly for $y$.
\end{lemma}
\begin{proof}
The interior  is
\begin{equation}\label{Pint}
\stackrel{\circ}{P}\,:= \Big\{ w\in W\big|\,w=\sum_{i=0}^{r}
b_i(w)w_i,\ \ \sum_{i=0}^{r}b_i(w)=1,\ \ b_i(w)>0\ \,\mathrm{for}\ 0\le i\le
r \Big\}.
 \end{equation}
Since $y\in P$ by assumption, $b_j(y)\ge0$ for $0\le j\le r$. Assume
now that $\overrightarrow{x ,y} $ pierces $P$.
Then for some $t_0\in[0,1]$ and all $j\in \{0,1,\dots,r\}$,
\begin{equation*}
b_j\big((1-t_0)x+t_0y\big)=(1-t_0)b_j(x)+t_0b_j(y)>0,
\end{equation*}
where we used \eqref{baryconvex}. If $b_j(y)=0$,
the above implies $b_j(x)>0$, as desired.

 Conversely, assume
$b_i(y)=0$ implies $b_i(x)>0$. If   $b_j(y)>0$, then  for some
$t_j<1$ and all $t_j\le t\le 1 $, we have $b_j\big((1-t)x+ty\big)>0$.
  If $b_j(y)=0$, so $b_j(x)>0$,
$$b_j\big((1-t)x+ty\big)=(1-t)b_j(x)+tb_j(y)=(1-t)b_j(x)
>0\qquad\qquad(0\le t<1).$$
Taking $s:=\max\{t_j\}<1$, we have $b_j\big((1-s)x+sy\big)>0$ for
all $j\in\{0,1,\dots,r\}$. Thus $(1-s)x+sy\in\,
\overrightarrow{x ,y} \,\cap \!\stackrel{\circ} {P} $, as claimed.

To prove the last part of the lemma, suppose
 $ z\in\ \stackrel{\circ} {P} $,
so $b_j( z)>0$ for $j\in\{0,1,\dots,r\}$. Then, for $0\le t<1$,
$$
b_j\big((1-t) z+ty\big)=(1-t)b_j( z)+tb_j(y)
\ge (1-t)b_j( z)>0,
$$
showing that $(1-t) z+ty\in \ \stackrel{\circ} {P} $, as claimed.
\end{proof}

The next lemma is  similar, so we omit the proof.
\begin{lemma} \label{conepiercing} Let $v_1,\dots,v_r$ be a basis of
the real vector space $W,$ let
$\displaystyle{C:=\sum_{j=1}^r\R_{\ge0}\cdot v_j}$
be an $r$-cone,   $y=\sum_{j=1}^r y_j v_j\in C\ \,
(\text{\ie }  y_j\ge0)$ and $x=\sum_{j=1}^rx_jv_j\in W$. Then
 $\overrightarrow{x,y} $ pierces $C$ if and
only if  $ x_j >0$ whenever  $ y_j=0\ \,(1\le j\le r).$
\end{lemma}
\subsection{Piercing and the $c_\sigma$'s} With notation
as in Proposition \ref{f},
 let us define $\overline{c}_\sigma \subset\R^{n-1}_+$ as
   the (closed) polytope with vertices
   $\varphi_{i,\sigma}:=\ell(f_{i+1,\sigma})\ \,(0\le i\le n-1)$,
\begin{equation}\label{csigmaclosure}
 \overline{c}_\sigma:=P(\varphi_{0,\sigma},
 \varphi_{1,\sigma},\dots,\varphi_{n-1,\sigma})
  =f(D_\sigma)=A_\sigma(D_\sigma) \qquad \qquad
   \big(\sigma\in S_{n-1}\big).
 \end{equation}
Our notation is somewhat misleading.
We define the polytope $\overline{c}_\sigma$
 for all $\sigma\in S_{n-1}$. However, we defined
  $c_\sigma\subset\overline{c}_\sigma$
 only when $w_\sigma\not=0$  $\big($see \eqref{csigma} and \eqref{wsigma}$\big)$.
 It will prove convenient to define
 $c_\sigma$ to be empty when $w_\sigma=0$.

\begin{lemma}\label{nondegenerate}
The polytope  $\overline{c}_\sigma$ defined in \eqref{csigmaclosure}
 is an $(n-1)$-simplex (\ie
 its $n$  vertices   are affinely independent)    if and only if
  $w_\sigma\not=0$. 
The affine map $A_\sigma$ in \eqref{Asigma}
  is invertible if and only if   $w_\sigma\not=0$.
\end{lemma}
\begin{proof} It suffices to prove that   $T_0,\dots,T_h\in\R^n_+$
are linearly independent if and only if $\ell(T_0),\dots,
\ell(T_h)\in\R^{n-1}_+$
are affinely independent. So suppose     $\ell(T_0),
\dots,\ell(T_h)\in\R^{n-1}_+$
are not affinely independent. Then for some $v\in\R^{n-1}$,
$$
v=\sum_{i=0}^h b_i \ell(T_i)=\sum_{i=0}^h b_i^\prime \ell(T_i),
\qquad \sum_{i=0}^h b_i =1=\sum_{i=0}^h b_i^\prime,
\qquad b_j\not=b_j^\prime\ \, \text{for some }j .
$$
Taking $x:=(v,1)\in\R^n$, we have $\ell(x)=v$ and, by Lemma \ref{Projection},
$$
 \sum_{i=0}^h \big(b_i/T_i^{(n)}\big) T_i=x =\sum_{i=0}^h \big(b_i^\prime/T_i^{(n)}\big) T_i ,
$$
showing that  the $T_i$ are not linearly independent.

Conversely, if $0= \sum_{i=0}^h c_i  T_i$ with some $c_j\not=0$, then
$$
T_j=\sum_{i=0}^h \delta^j_i  T_i=\sum_{i=0}^h \big(c_i+\delta^j_i \big)
  T_i\qquad(\delta^j_i:=0\text{ if }i\not=j,\ \delta^j_j :=1).
$$
But then Lemma \ref{Projection} shows that $\ell(T_j)$ has
two distinct sets of barycentric coordinates with respect to $\ell(T_0),\dots, \ell(T_h)$.

The final statement in the lemma follows from the last line of  \S\ref{PAF}.
\end{proof}

When $w_\sigma\not=0$,
 we  defined  in \eqref{csigma}
a set $c_\sigma$ lying between $\overline{c}_\sigma$ and its interior, \ie
 $\stackrel{\circ}{\overline{c}}_\sigma\,\subset c_\sigma\subset\overline{c}_\sigma$. If $\overline{c}_\sigma$
 has no interior, \ie $w_\sigma=0$, we   defined $c_\sigma=\varnothing$, the empty set.
Our next aim  is to  describe $c_\sigma$ in terms of piercing.

\begin{lemma}\label{csigmareason} For $z\in\R^{n-1}_+$, we have
 $z\in c_\sigma$   if and only if
  $\overrightarrow{0,z}$ pierces
$\overline{c}_\sigma$.
\end{lemma}
\begin{proof} If $z\notin  \overline{c}_\sigma$, then by definition  $\overrightarrow{0,z}$ does
not pierce $\overline{c}_\sigma$. As $c_\sigma\subset  \overline{c}_\sigma$, the lemma is clear in this case.
 If $w_\sigma=0$, there cannot be piercing as $\overline{c}_\sigma$
has an empty interior. Since  $c_\sigma=\varnothing$ when $w_\sigma=0$,
the lemma is also obvious in this case. Thus we may assume
 $z\in \overline{c}_\sigma$ and
 $w_\sigma\not=0$. We can then write $e_n:=[0,\dots,0,1]\in\R^n$
in the basis $f_{1,\sigma},\dots,f_{n,\sigma}$ of $\R^n$
as $e_n=\sum_{i=1}^n c_i(e_n) f_{i,\sigma},\ \,c_i(e_n)\in\R$.
 We have $e_n\in H_{i,\sigma}^+$ if and only if $c_i(e_n)>0$
(see footnote  \ref{Hfootie}).  Lemma \ref{Projection},
applied to $0=\ell(e_n)$, shows that
the barycentric coordinate $b_i(0)$ of 0 with
respect to the affine basis
 $\varphi_{0,\sigma},\dots,\varphi_{n-1,\sigma}$
 has the same sign as $c_{i+1}(e_n)\ \,(0\le i\le n-1)$.
 Thus $e_n\in H_{i+1,\sigma}^+$ if and only if $b_i(0)>0$.

Write $z\in \overline{c}_\sigma$ as
$z=\sum_{i=0}^{n-1}b_i(z)\varphi_{i,\sigma},\
b_i(z)\ge0,\ \sum_{i=0}^{n-1}b_i(z)=1$.
Suppose $z\in c_\sigma$. By  definition of
$c_\sigma\ \big($see \eqref{csigma}$\big)$,
 if $b_i(z)=0$, then $e_n\in H_{i+1,\sigma}^+$, \ie
    $b_{i}(0)>0$.    Lemma   \ref{piercing} now shows that
     $\overrightarrow{0,z}$ pierces  $\overline{c}_\sigma$.

      Conversely, if
$\overrightarrow{0,z}$ pierces  $\overline{c}_\sigma$, Lemma  \ref{piercing} shows
that if $b_i(z)=0$, then
$b_{i}(0)>0$, \ie $e_n\in H_{i+1,\sigma}^+$. Thus $z\in c_\sigma$.
\end{proof}

The next result
 justifies our description in \S \ref{Introduction}
of the cone $C_\sigma$ $\big($see \eqref{Csigma}$\big)$
in terms of piercing the closed cone
$\overline{C}_\sigma:=\sum_{i=1}^n\R_{\ge0}\cdot f_{i,\sigma}$ by a
line segment  from $e_n$.

\begin{lemma}\label{Csigmareason}  For $x\in\R^n_+$, we have  $x\in C_\sigma$ if and only if
$\overrightarrow{e_n,x}$ pierces $\overline{C}_\sigma$.
\end{lemma}
\begin{proof} As in the previous lemma, the cases  $w_\sigma=0$ or $x\notin\overline{C}_\sigma$ are trivial.  When $w_\sigma\not=0$, we can
write $e_n=\sum_{i=1}^n c_i(e_n)f_{i,\sigma}$, and
$x=\sum_{i=1}^n c_i(x)f_{i,\sigma}$. But $c_i(x)\ge0$, as $x\in\overline{C}_\sigma.$
Suppose $x\in C_\sigma$. Then, by \eqref{Csigma},
 $c_i(x)=0$ implies $e_n\in H_{i,\sigma}^+,$ \ie $c_i(e_n)>0$.
 Lemma \ref{conepiercing} shows then that
$\overrightarrow{e_n,x}$ pierces $\overline{C}_\sigma$.
Conversely, assume $\overrightarrow{e_n,x}$ pierces $\overline{C}_\sigma$.
Then, by Lemma \ref{conepiercing}, $c_i(x)=0$ implies
$c_i(e_n)>0,$ \ie $e_n\in H_{i,\sigma}^+$. But then $x\in C_\sigma$.
\end{proof}

The affine subspaces $h_{i,\sigma}$
 extending faces  of the polytope $\overline{c}_\sigma
 =f(D_\sigma)$ are
  \begin{equation}\label{defhisigma}
h_{i,\sigma}:=\Big\{\sum_{\substack{0\le j\le n-1\\
j\not=i}}b_j\varphi_{j,\sigma} \in\R^{n-1}\Big|\,b_j\in\R,
\ \sum_{\substack{0\le j\le n-1\\ j\not=i}}b_j =1 \Big\}
\qquad(0\le i\le n-1).
 \end{equation}
 We show next that none of these affine subspaces contains 0.
\begin{lemma}\label{goodpoint} The origin of $\R^{n-1}$ does not lie on any
   $h_{i,\sigma}$ $\,(0\le i\le n-1,\ \,\sigma\in S_{n-1})$.
\end{lemma}
\begin{proof}
Suppose otherwise. Then for some $i$ and $\sigma$ we have
$$
0=\sum_{\substack{0\le j\le n-1\\ j\not=i}}b_j\varphi_{j,\sigma}
     \qquad\qquad\Big(b_j\in\R,\ \,
     \sum_{\substack{0\le j\le n-1\\ j\not=i}} b_j=1\Big).
$$
Since  $\varphi_{j,\sigma}:=\ell(f_{j+1,\sigma})\ \,(0\le j\le n-1)$
 and $0=\ell(e_n)$, Lemma \ref{Projection}   applied to $x=e_n$
 and $T_j=f_{j+1,\sigma}\ (j\not=i)$,  shows
$$
e_n=\sum_{\substack{0\le j\le n-1\\ j\not=i}} c_j f_{j+1,\sigma}\qquad(c_j\in\R).
$$
This contradicts Lemma \ref{envantage}.
\end{proof}

\section{Maps between tori}
We will show  that $\R_+^{n-1}/{\widetilde{V}}$ is
homeomorphic to an $(n-1)$-torus and that the piecewise affine map
defined in Proposition \ref{f} descends to a map $F$ between  $(n-1)$-tori.
We then show that $F$ is homotopic to a homeomorphism $F_0$.

 To distinguish domains   we let
\begin{align}\label{LOGG}
&\LOG: \R^{n-1}_+\to \R^{n-1},\qquad \quad
\Big(\mathrm{LOG}\,\,x\Big)^{\!(i)}=
\log \,x^{(i)} \qquad (1\le i\le n-1),\\
\label{Logg}
&\Log: \R^{n}_+\to \R^{n-1},\qquad\qquad
\ \ \Big(\mathrm{Log}\,\,x\Big)^{\!(i)}=
\log\, x^{(i)} \qquad (1\le i\le n-1).
\end{align}
As in Theorem \ref{Main}, we assume given independent totally positive
 units $\varepsilon_1,\dots,\varepsilon_{n-1}$
in a totally real number field $k$ of degree $n\ge2$. We let
 ${\widetilde{\varepsilon}}_i=\ell(\varepsilon_i)\in\R^{n-1}_+\ \,\big($see \eqref{ell}$\big)$.

 We now relate the signed regulator of the $\varepsilon_i$
 to that of the ${\widetilde{\varepsilon}}_i$.
\begin{lemma}\label{REGS}
\begin{equation} \label{reg}
\det\!\big(\mathrm{LOG}  \,\,{\widetilde{\varepsilon}}_1,\dots, \mathrm{LOG}
\,\,{\widetilde{\varepsilon}}_{n-1} \big)=n\det\!\big(\mathrm{Log}\,\,\varepsilon_1,
\dots,\mathrm{Log}\,\,\varepsilon_{n-1}\big),
\end{equation}
where $\det(v_1,v_2,\dots,v_{n-1}) $ is the determinant of the $(n-1)\times
(n-1)$ matrix  having columns $v_i\in\R^{n-1}$.
In particular, neither of the above $(n-1)\times(n-1)$ determinants
 vanishes,   both have the same sign, and
 $\Lambda:=\sum_{i=1}^{n-1}\Z\cdot\LOG\,\,{\widetilde{\varepsilon}}_i \subset\R^{n-1}$ is a full lattice.
\end{lemma}
\begin{proof}
This is proved in \cite{DF}, but we repeat the proof here for completeness.
Using $ 1/\varepsilon_i^{(n)}= \prod_{j=1}^{n-1} \varepsilon_i^{(j)}$,
\eqref{reg} reduces to showing $ n=\det\big(I_{n-1}+
B_{n-1}\big), $ where the $(n-1)\times(n-1)$ matrices $I_{n-1}$ and
$B_{n-1}$ are, respectively, the identity and the matrix   whose
entries are all 1. But $\det\!\big(\lambda
I_{n-1}-B_{n-1}\big)=\lambda^{n-2}\big(\lambda-(n-1)\big)$, using
the obvious eigenvalues  $0$ and $n-1$ of $B_{n-1}$.
 Substituting $\lambda=-1$ concludes the
proof.
\end{proof}

By Proposition \ref{f}$\,$(iii), the map $f:I^{n-1}\to\R^{n-1}_+$  satisfies
$f(\sum_i b_ie_i)=\prod_i{\widetilde{\varepsilon}}_i^{\,b_i}$  on the vertices of the
hypercube, \ie when $b_i=0$ or 1 for all $i$. There is another map
$f_0:I^{n-1}\to\R^{n-1}_+$ that   trivially satisfies this on all of
$I^{n-1}$,
\begin{equation}\label{ff}
\big(f_0(x)\big)^{(j)}:=\prod_{i=1}^{n-1}
\big({\widetilde{\varepsilon}}_i^{(j)}\big)^{\,b_i}\qquad\qquad
 \Big(1\le j\le n-1,\ \,x= \sum_{i=1}^{n-1}
b_ie_i,\ \,0\le b_i\le1\Big).
\end{equation}
The map $f_0$  also satisfies (ii) of Proposition
\ref{f}, \ie
\begin{equation}\label{fff}f_0(x+e_i)={\widetilde{\varepsilon}}_if_0(x)\qquad\qquad\qquad
\big(  x\in I^{n-1}\ \mathrm{and}\   x+e_i\in I^{n-1}\big).
\end{equation}
On taking LOG
it is clear from Lemma \ref{REGS} that $f_0$ is the restriction to
$I^{n-1}$ of a homeomorphism between $\R^{n-1}$ and $\R^{n-1}_+$ (given by \eqref{ff}, but with $b_i\in\R$).

  Let
\begin{equation}\label{THat}
\widehat T\,:=\,I^{n-1}\!/\!\!\sim\ \,\simeq\ \R^{n-1}/\Z^{n-1}
\end{equation}
be the quotient space of $I^{n-1}$ by
the   closure of the  relation $x\sim x+e_i,\ $ whenever $x,x+e_i\in
I^{n-1}$. This is the usual  model of the standard torus
$\R^{n-1}/\Z^{n-1}$  as the cube $I^{n-1}$ with opposite points
identified. By Lemma \ref{REGS}, $\widehat T$ is
 homeomorphic to the   torus
\begin{equation}\label{T}
T:=\R^{n-1}_+/{\widetilde{V}}=\R^{n-1}_+/
\langle {\widetilde{\varepsilon}}_1,\dots,
{\widetilde{\varepsilon}}_{n-1}\rangle\simeq\R^{n-1}/\langle \mathrm{LOG}\,
{\widetilde{\varepsilon}}_1,\dots, \mathrm{LOG}\,
{\widetilde{\varepsilon}}_{n-1}\rangle.
\end{equation}
 The explicit
homeomorphism $F_0:\widehat T\to T$ is just the map induced by
  $f_0$ on the quotient tori. Part (ii) of Proposition \ref{f} insures that $f$ also induces
 a   continuous map $F:\widehat T\to T$.
 The situation is summarized in the   commutative diagrams
\begin{equation}\label{diagrams}
\begin{CD}
I^{n-1} @>f_0 >> \R^{n-1}_+\\
@VV\widehat{\pi}V @VV\pi V\\
\widehat{T} @>F_0>\simeq> T
\end{CD}\  \qquad
\qquad \qquad
\quad\begin{CD}
I^{n-1} @>f>> \R^{n-1}_+\\
@VV\widehat{\pi}V @VV\pi V\\
\widehat{T} @>F>{\phantom{\simeq}}> T
\end{CD}\
\end{equation}
where $\widehat{\pi}$ and $\pi$  are the natural quotient maps.

The set $f_0\big([0,1)^{n-1}\big)\subset\R^{n-1}_+$  is an obvious fundamental domain
for the action of $\widetilde{V}$ on $\R^{n-1}_+$.
 We will   show  in \S\ref{finalsection} that this  fundamental domain with curved boundaries
can be deformed  by a homotopy  into
a signed fundamental domain composed of (partly closed) polytopes. The first step towards
proving this is to find a homotopy between the maps $F$ and $F_0$ on the tori.
\begin{lemma}\label{tori}Suppose $g_0$ and $g_1$ are
 continuous maps from $I^{n-1}:=[0,1]^{n-1}$
to $\R^{n-1}_+$  such that for any standard basis vector
 $e_j$ of $\R^{n-1}$,
$g_i(x+e_j) ={\widetilde{\varepsilon}}_j\cdot  g_i(x) $
  whenever $ x\in I^{n-1}$ and $ x+e_j\in
I^{n-1}\ \,( i=0,\,1).$
  Let $G_i:\widehat T\to  T$    be the map induced by $g_i$ on the
  tori defined in \eqref{THat} and \eqref{T}.
     Then    $G_0$ is homotopic to $G_1$. In particular, the maps
$F :\widehat T\to T$  and
$F_0:\widehat T\to T$ between $(n-1)$-tori defined by \eqref{diagrams}  are homotopic.
\end{lemma}
\begin{proof}
 For $0\le t\le1$, define $g_t:I^{n-1}\to
\R^{n-1}_+$ by $g_t(x):=(1-t)g_0(x)+tg_1(x)$. Clearly,
$(t,x)\to g_t(x)$ is continuous. If $ x\in I^{n-1}$ and $x+e_j\in
I^{n-1},$ then
$$
g_t(x+e_j)=(1-t)g_0(x+e_j)+tg_1(x+e_j)=(1-t){\widetilde{\varepsilon}}_j\cdot g_0(x)
+t{\widetilde{\varepsilon}}_j\cdot g_1(x)={\widetilde{\varepsilon}}_j\cdot g_t(x).
$$
Thus $g_t$ descends to a homotopy $G_t:\widehat T\to T$ between
$G_0$ and $G_1$.
\end{proof}

\section{Review of topological degree theory}

Algebraic topology gives  an elegant approach to degree theory using
homology groups. More elementary (homology-free, but far longer) treatments of
degree theory \cite[\S III]{OR} first define the degree of a proper smooth   map at regular values,
 and then
apply an approximation process to define the degree of a proper continuous
map.
  Our application of degree
theory in \S\ref{finalsection} will concern the local and global
degrees of the map $F:\widehat{T}\to T$
  in  \eqref{diagrams}. This map is proper and continuous, but not
everywhere differentiable. However, every point of $T$ is the
limit of regular values of $F$, so we will still
be able to compute the local and global degrees of $F$.

There are several textbooks devoted entirely to degree theory, but we shall only need to
draw on a few pages   of Dold's algebraic
topology textbook \cite[pp.\ 266--269]{Dol}.
These pages rely on   basic  singular  homology theory, such as
 excision and homotopy invariance \cite[Ch.\ II--III, pp.\ 16--46]{Dol}
 \cite[\S8--15, pp.\ 35--68]{Gre},
and the   calculation
    of the relative singular homology group
 \cite[Ch.\ VIII, \S2.6, 3.3,  3.4]{Dol}  \cite[\S22,  Cor.\ 22.26, p.\ 121]{Gre}
\begin{equation}\label{homology}  
H_r(M,M-C)\cong
  \begin{cases}
      \Z^t&\text{ if }C\text{ is compact and has $t$ connected components}, \\
      0&\text{ if } C \text{ is connected, but not compact.}
  \end{cases}
\end{equation}
Here $M$ is an orientable $r$-dimensional manifold
  and  $H_r(Y,X)=H_r(Y,X;\Z)$ denotes
  the $r^{\text{th}}$
relative singular homology group with $\Z$-coefficients
 of the topological space $Y$ mod
 its subspace $X$  \cite[Ch.\ III, \S3.1]{Dol} \cite[\S 13]{Gre}.  This fact
underlies the  definition in \S\ref{Basicdegree} below  of the topological
  degree in terms of the fundamental
class of a compact set and explains the crucial
 local-global principle in Proposition \ref{DegreeProps} \eqref{Localglobal} below. In particular,
 if $P\in M$ we have $H_r(M,M-P)\cong\Z$ (but this has an easy proof \cite[Ch.\ VIII, \S2.1]{Dol}).

 An isomorphism of
homology groups (always taken with $\Z$-coefficients) will sometimes be written $\widetilde{\to}$ or
$\widetilde{\leftarrow}$ to indicate
that it is induced by an inclusion of topological spaces. By an
 $r$-manifold  $M^r$ we  mean  an $r$-dimensional topological manifold
 without boundary.  Our  manifolds will all have the same fixed dimension $r$, so we
 often write $M$ for $M^r$.

\subsection{Basic properties}\label{Basicdegree}

If  $M$ is an $r$-manifold and $P\in M$, we will write
 $o_P$ for a choice of one of the two generators of
  $H_r(M,M-P)\cong\Z$.
   We will assume that all our manifolds are orientable and oriented,
  \ie we assume given a consistent (``locally constant") choice of
   $ o_P= o_P(M)$ for all $P\in M$
  \cite[Ch.\ VIII, \S2.9]{Dol}. An oriented
   open subset $W\subset M$ has the orientation induced from $M$
     if   for all $P\in W$,   the isomorphism
    $H_r(W,W-P)\, \widetilde{\to}\, H_r(M,M-P)$ maps $ o_P(W)$
        to $ o_P(M)$. We will  call such a $W$
    an (oriented) $r$-submanifold.
 If $M$ is orientable and connected,
  an orientation on an open subset $W$ determines a unique
   orientation on $M$, \ie the one for which the given
    orientation on $W$ coincides with the one
  induced from $M$. In fact, on a connected
 orientable $r$-manifold $M$,
    a    generator $ o_P\in H_r(M,M-P)$ for a single $P\in M$ determines a
     unique orientation on $M$ satisfying $ o_P(M)=o_P$.

More generally, for a compact  non-empty subset $K\subset M $
of an (oriented)   $r$-manifold $M$,
 the fundamental
class $ o_K= o_K(M)$ of $K$ can be characterized as
the unique element of $ H_r(M,M-K)$ mapping to
$ o_P(M)\in H_r(M,M-P)$ for every $P\in K$ \cite[Ch.\ VIII, \S4.1]{Dol}.
 Here the map on
homology is induced by the
inclusion of pairs $(M,M-K)\to(M,M-P)$.
If $K$ is empty,  $ o_K:=0$.
If $K$ is connected and not empty, $ o_K$ is a
 generator of  $ H_r(M,M-K)\cong\Z$ \cite[Ch.\ VIII, \S4.1]{Dol}.

If $G:N\to M$ is a  continuous   map between two
oriented $r$-manifolds and $K\subset M$ is  connected, non-empty
and $G^{-1}(K)\subset N$ is compact,
we define the degree of $G$ over $K$
as the unique integer $\deg_K(G)$ such that
the  induced map on 
homology $G_*:H_r\big( N,N-G^{-1}(K)\big)\to H_r\big(M,M-K\big)$
  satisfies
\begin{equation}\label{DegreeK}
G_*\big( o_{G^{-1}(K)}\big)=\deg_K(G)\cdot o_K\,.
\end{equation}
Often, instead of listing     the above assumptions on $K$ and $G$, we shall simply say that
$\deg_K(G)$ is defined. Note that if $N=M$ (with the same orientation) and Id is the
identity map, then $\deg_{K}(\text{Id})=+1$.

We now give the main properties of  the topological degree. Some of these
  obviously follow from the others, but we give them anyhow for later reference.

\begin{proposition}\label{DegreeProps}  Suppose $G:N\to M$ is a
 continuous  map between two
oriented $r$-manifolds, and suppose $K\subset M$ is
a connected,  non-empty  compact subset of $M$ with
  $G^{-1}(K)\subset N$   compact. Then $\deg_K(G)$ is defined and the following hold.
  \begin{enumerate} 
  \item \label{Subsets} {\bf{(Degree over subsets)}}  If $I\subset K$ is
a connected, non-empty compact subset of $K$, then  $\deg_I(G)$ is  defined and
$\deg_I(G)= \deg_K(G)$.
  \item\label{Points} {\bf{(Shifting  points)}}
  If $P$ and $Q$ are points in  $K$, then $\deg_P(G)$
   and $\deg_Q(G)$ are defined and
  $\deg_P(G) =\deg_Q(G)$.
  \item \label{NotSurj} {\bf{(Maps missing a point of $K$)}}
    If $K\not\subset G(N)$, then  $\deg_K(G)=0$.
  \item \label{Homotopy} {\bf{(Homotopy invariance)}}
     Suppose $\Theta:N\times[0,1]\to M$
  is continuous   and   $\Theta^{-1}(K)\subset N\times[0,1]$
   is compact. Define $\Theta_t:N\to M$ as
    $\Theta_t(n):=\Theta(n,t)$ and suppose $G=\Theta_0$.
   Then $\deg_K(\Theta_1)$ is defined   and
   $\deg_K(G)=\deg_K(\Theta_1)$.
  \item\label{Propermaps} {\bf{(Global degree
   for proper maps)}}  If $G$ is  proper
   (\ie $G^{-1}(L)\subset N$ is compact for any
   compact $L\subset M$) and $M$ is connected,
  then $\deg_L(G)$ is defined for any connected,
   non-empty compact subset $L$ of $M$,
  and $\deg_L(G)=\deg_K(G)$. We let $\deg(G):=
  \deg_L(G)$ for any non-empty compact subset $L\subset M$.
  \item \label{CompactCase} {\bf{(Compact case)}}
   Suppose $N$ is compact and $M$ is connected. Then
  $\deg_L(G)$ is defined for any  non-empty,
  connected compact subset   $L\subset M$.
  Moreover, if $ G^\prime:N\to M$ is homotopic to $G$, then $\deg(G)= \deg(G^\prime)$.
   \item\label{Composing}  {\bf{(Composition)}}   If $N^\prime$ is
    an oriented $r$-manifold,  $g:N^\prime\to N$ is proper and $N$ is connected, then
    $\deg_K(G\circ g)$ is defined and   $\deg_K(G\circ g)=
    \deg_K(G)\cdot\deg(g)$.
     \item\label{Homeomorphisms} {\bf{(Homeomorphisms)}}
      If $G$ is  a homeomorphism between connected manifolds,
 then $\deg(G) =\pm1$. In fact, $\deg(G)=+1$ if and only if
  $G$ is orientation-preserving, \ie $G_* \big(o_P(N)\big)
  =o_{G(P)}(M)$ for some (and therefore any) $P\in N$.
     \item\label{Localglobal}  {\bf{(Local-global)}}
     Suppose $U_i\subset N\ (1\le i\le t)$ are
      $r$-submanifolds
    (\ie open subsets with the induced orientation) such that
    $$
    G^{-1}(K)\subset \bigcup_{i=1}^t U_i,\qquad\qquad U_i\cap U_j\cap
     G^{-1}(K)=\varnothing\qquad(  i\not= j).
    $$
    Let  $G_{U_i}$ denote  $G$ restricted to $U_i$.
   Then $\deg_K(G_{U_i})$ is defined and
    $$
    \deg_K(G)=\sum_{i=1}^t\deg_K(G_{U_i}).
    $$
    \item\label{Shrinking}  {\bf{(Shrinking the domain)}}
    Assume      $G^{-1}(K)\subset U$, where  $U\subset N$ is an $r$-submanifold,
    and  let $G_U$ denote  $G$ restricted to $U$. Then $\deg_K(G_U)$ is defined
    and $ \deg_K(G_U)=\deg_K(G)$.
  \end{enumerate}
\end{proposition}
\begin{proof}
Claims \eqref{Subsets}, \eqref{Points} and \eqref{NotSurj}
 are proved in \cite[Ch.\ VIII, \S4.4]{Dol}.
To prove \eqref{Homotopy}  \cite[Ch.\ VIII, \S4.10, Exercise 3]{Dol}, let
$$
K^\prime:=\big\{ n\in N  \big|\,  \Theta(n,t)\in K\ \text{for some }t\in[0,1]\big\}
$$
be the projection to
$N$ of the compact set $\Theta^{-1}(K)\subset N\times[0,1]$.
Thus $K^\prime\subset N$ is compact and    $\Theta_t^{-1}(K)\subset K^\prime\ \,(0\le t\le1)$. Hence
  $\Theta$ gives a homotopy of pairs $\Theta_t:(N,N-K^\prime)\to(M,M-K)$.
   Passing to  homology, by   homotopy invariance \cite[Ch.\ III, \S5.2]{Dol},
$$
{\Theta_0}_*={\Theta_1}_*.
$$
 Since
 $\Theta_t^{-1}(K)\subset K^\prime $    is a closed subset of a compact set,  $\Theta_t^{-1}(K)$
 is compact and so $\deg_K(\Theta_t)$ is defined.  Also \cite[Ch.\ VIII, \S4.3]{Dol},
$$
{\Theta_t}_*\big(o_{ K^\prime}(N) \big)=\deg_K(\Theta_t)\cdot o_K(M).
$$
Combining the last two displays, we have
$$
\deg_K(\Theta_0)\cdot o_K(M)={\Theta_0}_*\big( o_{K^\prime}(N) \big)=
 {\Theta_1}_*\big( o_{K^\prime}(N) \big)=\deg_K(\Theta_1)\cdot o_K(M).
$$
Since  $G=\Theta_0$,  we find $\deg_K(G)=\deg_K(\Theta_0)=\deg_K(\Theta_1)$, proving \eqref{Homotopy}.

Claims \eqref{Propermaps} and  \eqref{Composing} are
 proved in  \cite[Ch.\ VIII, \S4.5--4.6]{Dol}.
 Claim  \eqref{Homeomorphisms} follows from \eqref{Propermaps}, with $L:=G(P)$. To prove
 \eqref{CompactCase},   note that
 any  continuous  map from a compact manifold is proper. Also,
  if $\Theta$ is a homotopy between
 $G$ and $G^\prime$, and  $Q\in M$, then
    $\Theta^{-1}(Q)\subset [0,1]\times N$ is compact, as $N$ is assumed compact.
     From  \eqref{Propermaps} and \eqref{Homotopy},
     $\,\deg(G)=\deg_Q(G)=\deg_Q(G^\prime)=\deg(G^\prime)$, proving \eqref{CompactCase}.

To prove \eqref{Localglobal},     let $U_0:= N-G^{-1}(K)$. Then $U_0$ is
  an open subset of $N$ and
    $U_0\cap U_i\cap G^{-1}(K)=\varnothing$ for $i\not=0$. Also,
     $ \bigcup_{i=0}^t U_i= N$, so by \cite[Ch.\ VIII, \S4.7]{Dol},
$$
\deg_K(G)=\sum_{i=0}^t\deg_K(G_{U_i}).
$$
But $\deg_K(G_{U_0})=0$ by \eqref{NotSurj}, as
 $G(U_0)\not\subset K$, so we have proved \eqref{Localglobal}.

Claim \eqref{Shrinking} follows from \eqref{Localglobal} with $t=1$.
\end{proof}

\subsection{Local degree}\label{Local} Suppose $G:N\to M$ is a
map between oriented manifolds
and that  $p\in N$ is an isolated point of $G^{-1}\big(G(p)\big)$.
Thus there is an $r$-submanifold
$V\subset N$ (\ie an open subset with the induced orientation)
 such that $G^{-1}\big(G(p)\big)\cap V=\{p\} $. Then
$\deg_{G(p)}(G_V)$ is defined, where $G_V$ is $G$ restricted to $V$.
If $V^\prime\subset N$ is another  $r$-submanifold
such that $G^{-1}\big(G(p)\big)\cap V^\prime=\{p\} $,
 Proposition \ref{DegreeProps} \eqref{Shrinking}
shows
$$
\deg_{G(p)}(G_V)= \deg_{G(p)}(G_{V\cap V^\prime})=\deg_{G(p)}(G_{ V^\prime}).
$$
Hence $\deg_{G(p)}(G_V)$ depends only on $p$ and $G$, so we shall write
\begin{equation}\label{LocDeg}
 \mathrm{locdeg}_p(G ):= \deg_{G(p)}(G_V)\qquad\qquad
 \big(p\in N,\ \,V\cap G^{-1}\big(G(p)\big)=\{p\}\big),
\end{equation}
and call  $ \mathrm{locdeg}_p(G )$ the  local degree of $G$ at $p$.

If $G$ is a local homeomorphism at $p$
(\ie $G$ restricted to some open neighborhood  of
$p$   is a homeomorphism onto its image),
then $\mathrm{locdeg}_p(G )$ is certainly defined and equals $\pm1$
by  Proposition \ref{DegreeProps} \eqref{Homeomorphisms}.
If $G:N\to M$ is a homeomorphism between connected manifolds,
 we have by  Proposition \ref{DegreeProps} \eqref{Propermaps} and \eqref{Shrinking},
\begin{equation}\label{LocGlobDeg}
 \mathrm{locdeg}_p(G ) = \deg(G)\qquad\qquad(p\in N).
\end{equation}
 If $g:N^\prime\to N$ and $G:N\to M$ are
local homeomorphisms at $p^\prime\in N^\prime$ and
at $g( p^\prime )\in N$ respectively, then  Proposition \ref{DegreeProps} \eqref{Composing} shows that
\begin{equation}\label{ComposeLocDeg}
 \mathrm{locdeg}_{p^\prime}(G\circ g ) = \mathrm{locdeg}_{g(p^\prime)}(G )
 \cdot  \mathrm{locdeg}_{p^\prime}(g ) .
\end{equation}

We   now    prove the standard
 formula for the   degree of a local diffeomorphism of Euclidean space.
 This formula is often taken as the starting point for the definition of
the local degree of a map between smooth oriented $r$-manifolds. We include a proof since
Dold \cite{Dol} does not treat the differentiable case.
\begin{proposition}\label{localdegree} Fix an orientation on $\R^r$ and give the
open subset $U\subset\R^r$   the induced orientation. Let  $G:U\to \R^r$ be
  continuously  differentiable  and suppose that at some $\gamma\in U$, the differential
 $d G_\gamma:\R^r\to\R^r$ of $G$ at $\gamma$ is an invertible linear transformation.
 Then $\mathrm{locdeg}_{\gamma}(G)$ is defined and
\begin{equation}\label{detformula}
  \mathrm{locdeg}_{\gamma}(G)
 =\mathrm{sign}\big(\!\det(d  G_\gamma)\big).
\end{equation}
\end{proposition}
\noindent We note quite  generally, that if $G:U\to M$ and $U\subset M$ is
 an $r$-submanifold (with the induced orientation) of
  an oriented $r$-manifold $M$, then
$\deg_K(G)$ (when defined) is independent of the
 choice of orientation on $M$. Hence it is
not surprising that $\mathrm{locdeg}_{\gamma}(G)$ in
 \eqref{detformula} is independent of the orientation on $\R^r$.
\begin{proof}
We first compute
the degree of a translation $T_\alpha:\R^r\to\R^r$ given
 (for some fixed $\alpha\in\R^r$) by
$T_\alpha(v):=v+\alpha$ for $v\in\R^r$. Let $\Theta:[0,1]\times \R^r\to\R^r$ be
defined by $\Theta(t,v):=v+t\alpha$. Then
 $\Theta^{-1}(0)=\big\{(t,-t\alpha)\big|\,t\in[0,1] \big\}$,
  a compact set. Hence, by  Proposition \ref{DegreeProps} \eqref{Homotopy},
$$
\deg_0(T_\alpha)=\deg_0(\Theta_1)= \deg_0(\Theta_0)=\deg_0(\text{Id})=+1\qquad\quad\big(\Theta_t(v):=\Theta(t,v)\big),
$$
\ie translations have (global and local)
 degree $+1$, in agreement with \eqref{detformula}.

 Next we consider an invertible linear function $T:\R^r\to\R^r$ and prove that
  \begin{equation}\label{lineardegree}
 \mathrm{locdeg}_\gamma( T)=\mathrm{sign}\big(\!\det(T)\big).
  \end{equation}
  If $\det(T)>0$,
  there is a continuous path $T_t\in \mathrm{GL}(r,\R)$
  connecting $T=T_0$ to the identity map $\text{Id}=T_1$.
  In Proposition \ref{DegreeProps} \eqref{Homotopy},
  let $\Theta:[0,1]\times \R^r\to\R^r$ be
defined by $\Theta(t,v)=T_t(v)$. Then $\Theta^{-1}(0)
=\big\{(t,0)\big|\,t\in[0,1] \big\}$,
 a compact set. Hence
 $ \deg(T)=\deg(T_0)=\deg(T_1)=\deg(\text{Id})=+1$. If  $\det(T)<0$,
   there is a continuous path $T_t\in \mathrm{GL}(r,\R)$
  connecting $T$ to a reflection $T_1$ across a hyperplane though 0. Here
    an explicit calculation with simplices   shows  that
    $\deg(T_1)=-1$    \cite[Ch.\ IV, \S4.3]{Dol}.

We can now prove Proposition \ref{localdegree}.
 By the inverse function theorem, $G$ is a
local diffeomorphism in some
 neighborhood of $\gamma$, hence the local
degree of $G$ is certainly defined at $\gamma$.
In view of \eqref{ComposeLocDeg}, after composing with translations
 we can assume that  $\gamma=0$ and $G(0)=0$.
 Since we already know \eqref{detformula} for   linear maps, by
 considering $dG_0^{-1}\circ G$, \eqref{ComposeLocDeg}
  shows that we may  assume $dG_0=\mathrm{Id}$.
 After these simplifications, the proposition will be proved once we have
$\mathrm{locdeg}_0( G)= +1$.

 To calculate the local degree of $G$ at 0,
  we may  restrict $G$ to any small enough open ball
  $B :=\big\{x\in\R^r\big|\,\|x\|<\delta \big\} $,
where  $\|\ \|$ denotes  the Euclidean norm on $\R^r$.
Since $G$ is
  differentiable at 0 and $dG_0=\mathrm{Id}$,  for some $\delta>0$ we have
$ G(x)=x+w(x),$ where $ \|w(x)\|\le \|x\|/2$ for $x\in B$. Also, $w(0)=0=G(0)$.
Define  $\Theta:[0,1]\times B\to \R^r$   by $\Theta(t,x):=x+tw(x)$,
 so that $\Theta_0(x):=\Theta(0,x)=x$ and $\Theta_1(x):=\Theta(1,x)=G(x)$ for $x\in B$.
Note that $\Theta(t,x)=0$ if and only if $x=0$ since, for $  x\in B$ and $0\le t\le1$,
 $$  \| \Theta(t,x)\|=\|x+tw(x)\|\ge \|x\|-\|tw(x)\|\ge
\|x\|-\|x\|/2>0\qquad(x\not=0).
$$
Hence $\Theta^{-1}(0)=\big\{(t,0,)\big|\,t\in[0,1]\big\}$, a compact set,
and so homotopy invariance   gives
$ \mathrm{locdeg}_0( G)=\mathrm{locdeg}_0( \text{Id})=+1,$
as claimed.\end{proof}

\section{Proof of main theorem}\label{finalsection}

We have shown (see Proposition \ref{conestosimplices} and Definition \ref{vd}) that to prove   Theorem \ref{Main},
  we need to prove the basic count
 \begin{equation}\label{AIM}
\sum_{\substack{\sigma\in S_{n-1}\\ w_\sigma\not=0}}
\,\sum_{z\in c_\sigma\cap\widetilde{V}\cdot y}w_\sigma=1\qquad\qquad\big(  y\in\R^{n-1}_+\big),
\end{equation}
 and that   the number of elements of $c_\sigma\cap {\widetilde{V}}\cdot y$ is
   bounded independently of $y$. This latter part  is clear  on applying the
    isomorphism $\LOG:\R_+^{n-1}\to\R^{n-1}$. Indeed, $\LOG(c_\sigma)$ has
    closure  $\LOG(\overline{c}_\sigma)$, a compact set, and $\LOG({\widetilde{V}})$ is a  lattice (see Lemma
    \ref{REGS}).

    We will prove \eqref{AIM} by showing  that it is   an instance
 of the local-global principle applied to the map $F:\widehat{T}\to T$  defined in \eqref{diagrams}.
 It  is not hard to calculate the global degree of $F$ since $F$ is
  homotopic to the much simpler map $F_0$, also  defined in \eqref{diagrams}. The
  calculation  of the local degree of $F$ at a  generic point will   prove straight-forward, yielding
 \eqref{AIM} for a generic $y\in\R^{n-1}_+$.

  To deal with the remaining $y$ (those whose $\widetilde{V}$-orbit
 intersects a boundary piece of some $\overline{c}_\sigma$), we will
 approach $y$ along the segment  $\overrightarrow{0,y} $
and show that   its  points  are generic when they are sufficiently close to $y$.
 This will allow  us to conclude that \eqref{AIM} also holds for $y$.

\subsection{Global degree} We fix once and for
all an orientation of $\R^{n-1}$  and use it to fix orientations
 on the $(n-1)$-tori $\widehat{T}$ and $T$   in \eqref{THat} and \eqref{T}  as follows.
  Since $\widehat\pi: [0,1]^{n-1}\to \widehat T$ restricted to
$(0,1)^{n-1}$ is a  local homeomorphism and tori are connected and orientable, we
orient $\widehat T$ by declaring $\widehat\pi$ to be orientation-preserving.
 Here the open subset $(0,1)^{n-1} \subset\R^{n-1}$ is given
the induced  orientation. Thus the local degree of $\widehat\pi$
at any point of $(0,1)^{n-1}$ is
$+1$.  Similarly, we orient $T=\pi(\R^{n-1}_+)$ by giving $\R^{n-1}_+ \subset\R^{n-1}$ the
  induced orientation,    declaring the local homeomorphism
  $\pi:\R^{n-1}_+ \to T$ to have local degree
$+1$.

\begin{lemma}\label{degF} Let
$\varepsilon_1,\dots,\varepsilon_{n-1}$ be independent totally
positive units of a totally real field $k$ and let $F:\widehat{T}\to T$
be as defined in \eqref{diagrams}.  Then $\deg(F)$ is defined and
\begin{equation}\label{degreeFlemma}
\deg(F)=\mathrm{sign}\big(\!\det(\Log\,\,\varepsilon_1,
\Log\,\,\varepsilon_2,\dots,\Log\,\,\varepsilon_{n-1})\big)=\pm1,
 \end{equation}
 where $\det(v_1,v_2,\dots,v_{n-1}) $ is the determinant of the $(n-1)\times
(n-1)$ matrix  having columns $v_i\in\R^{n-1}$, and $\Log:\R^n_+\to\R^{n-1}$ is given by
$\big(\Log\, \,x\big)^{(j)}=\log\, x^{(j)} \ \,(1\le j\le n-1)$.
\end{lemma}
\noindent Before proving the lemma, we note that   $\deg(F)=\pm1\not=0$ implies that $F$ is surjective
 (use Proposition \ref{DegreeProps} \eqref{NotSurj} with
$K=N=T$ and $G=F$). Since $\widehat{\pi}$ is surjective, we see from \eqref{diagrams} that
 $\pi \circ f=\widehat{\pi}\circ F$ is also surjective.
Since the image of $f$ is the union of the polytopes $\overline{c}_\sigma$, this means that every
  orbit $\widetilde{V}\cdot y\subset\R^{n-1}_+$ must intersect at least one $\overline{c}_\sigma$, \ie
   $\bigcup_{\sigma\in S_{n-1}}\overline{c}_\sigma$ contains a true
   fundamental domain for  $\widetilde{V}$ acting on $\R^{n-1}_+$.
\begin{proof}
By Lemma \ref{tori}, $F$ and $F_0$ are homotopic maps between compact,
 connected, oriented manifolds. By Proposition \ref{DegreeProps} \eqref{CompactCase},
 their degrees are defined and
$\deg(F_0)=\deg(F)$.
Hence it suffices to show that $\deg(F_0)$ is  given by the sign of the determinant in  \eqref{degreeFlemma}.

 Since
$F_0$ is a homeomorphism of connected manifolds $\big($see \eqref{diagrams}$\big)$,
 \eqref{LocGlobDeg} shows
  $$
  \deg(F_0)=\mathrm{locdeg}_{\widehat\pi(P)}(F_0)
  $$ for any $P\in (0,1)^{n-1}$. By \eqref{diagrams},
$F_0\circ\widehat\pi=\pi\circ f_0$, and $f_0$ is a local
homeomorphism around $P$. Recall that $\pi:\R^{n-1}_+\to T$ is a local homeomorphism everywhere
and $\widehat{\pi}:[0,1]^{n-1}\to \widehat{T}$ is a local homeomorphism at all $P\in(0,1)^{n-1}$. Since
 we have oriented $\widehat T$ and $T$ so that the local degree
of $\widehat\pi$ and $\pi$ is $+1$, by \eqref{ComposeLocDeg} we have
$$
\deg(F_0)=\mathrm{locdeg}_{\widehat\pi(P)}(F_0)=\mathrm{locdeg}_P (f_0)\qquad\qquad\big(P\in(0,1)^{n-1}\big).
$$
To compute the
latter degree note that   by Proposition \ref{localdegree}, the
diffeomorphism   $\LOG:\R^{n-1}_+\to\R^{n-1}$
in \eqref{LOGG} has local degree $+1$.
Thus
$$
\deg(F_0)=\mathrm{locdeg}_P (f_0)= \mathrm{locdeg}_P
(\LOG\circ f_0)=\mathrm{sign}\big(\!\det(\LOG\,\,{\widetilde{\varepsilon}}_1,
\dots,\LOG\,\,{\widetilde{\varepsilon}}_{n-1})\big),
$$
 since $\LOG\circ f_0:\R^{n-1}\to\R^{n-1}$ is an invertible
 linear map
 taking the basis element $e_i$ to $\LOG\,\,{\widetilde{\varepsilon}}_i\,$
 (again use Proposition  \ref{localdegree}).
   Lemma \ref{REGS} shows that the above
 determinant has the same sign if we replace
  $\LOG\,\,{\widetilde{\varepsilon}}_i$ by $\Log\,\,\varepsilon_i$.
\end{proof}

\subsection{Proof of the basic count for generic points}\label{easycase}
We first calculate
the local degree of   $F$
 at points where it  is a local diffeomorphism.
\begin{lemma}\label{goodcasedegree} If $x$ is an
interior point of the simplex $D_\sigma$ and $w_\sigma\not=0\,$
 $\big($see \eqref{Dsigma} and \eqref{wsigma}$\big)$,
then   the local degree
 of $F$ at $\widehat\pi(x)$ is defined    and
\begin{equation}\label{vsigma}
\mathrm{locdeg}_{\widehat\pi(x)}(F) =v_\sigma:=(-1)^{n-1}\mathrm{sgn}(\sigma)
\cdot\mathrm{sign}\big(\!\det(f_{1,\sigma},f_{2,\sigma},
\dots,f_{n,\sigma})\big),
 \end{equation}
 where $\mathrm{sign}\big(\! \det(v_1,v_2,\dots,v_n) \big)$
 is the sign of the
determinant of the $n\times n$ matrix
 having columns $v_i\in \R^n$,
  and $\mathrm{sgn}(\sigma)=\pm1$
 is the sign  of the permutation
$\sigma\in S_{n-1}$.
\end{lemma}
\begin{proof} Recall from  \eqref{diagrams} that
 $F\circ\widehat\pi=\pi\circ f$, with $f$   as in Proposition \ref{f}.
  Since $f$  restricted to $D_\sigma$
  is   the affine map $A_\sigma$, which by Lemma \ref{nondegenerate} is a
   bijection when $w_\sigma\not=0$, it is clear that $f$ is a local diffeomorphism
    around $x$. But $\widehat\pi$ and $\pi$ are also local
     diffeomorphims of degree $+1$,  so  $F$ is a local
     diffeomorphism around $\widehat{\pi}(x)$.
     By \eqref{ComposeLocDeg} and Proposition \ref{localdegree},
\begin{equation}\label{degAsigma}
\mathrm{locdeg}_{\widehat\pi(x)}(F)=\mathrm{locdeg}_x(f)
=\mathrm{sign}\big(\!\det(L_\sigma)\big),
 \end{equation}
where $L_\sigma:\R^{n-1}\to\R^{n-1}$ is the linear part of
$A_\sigma$. In the basis $\{\phi_{i,\sigma}\}_{i=1}^{n-1}$ of
$\R^{n-1}$,
$$
L_\sigma(\phi_{i,\sigma})=A_\sigma(\phi_{i,\sigma})-A_\sigma(0)=
\varphi_{i,\sigma}-\varphi_{0,\sigma}= \ell(f_{i+1,\sigma}) \,-
\,1_{n-1}\ \quad\big(1_{n-1}:=(1,1,\dots,1)\big),
$$
where we used   \eqref{phiisigma}, \eqref{Asigma}, $\phi_{0,\sigma}:=0$ and  the paragraph
following \eqref{affinesubspaces}.

We now compute $\det(L_\sigma)$. Let
$\{e_i\}_{i=1}^{n-1}$ be the standard basis of $\R^{n-1}$.  From \eqref{phiisigma},
$\phi_{i,\sigma}:=\sum_{j\le i}e_{\sigma(j)}$, so
$$
L_\sigma(e_{\sigma(i)})=L_\sigma(\phi_{i,\sigma}-
\phi_{i-1,\sigma})=L_\sigma(\phi_{i,\sigma})-
L_\sigma(\phi_{i-1,\sigma})=\varphi_{i,\sigma}-
\varphi_{i-1,\sigma}\quad(1\le i\le n-1).
$$
Let $P_\sigma:\R^{n-1}\to\R^{n-1}$ be the linear map determined by
$P_\sigma(e_i):=e_{\sigma(i)}$, so that
$\det(P_\sigma)=\mathrm{sgn}(\sigma)$. We have just shown that
$$
 \mathrm{sgn}(\sigma)\det(L_\sigma)=\det(L_\sigma\circ
 P_\sigma)=\det\!\big(\varphi_{1,\sigma}-\varphi_{0,\sigma},
 \varphi_{2,\sigma}-\varphi_{1,\sigma},
 \dots,\varphi_{n-1,\sigma}-\varphi_{n-2,\sigma}\big).
 $$
 Adding the first column above to the second, then the second to the third
 and so on, we find using $\varphi_{0,\sigma}=1_{n-1}$,
 \begin{equation}\label{detLsigma}
\mathrm{sgn}(\sigma)\det(L_\sigma)=
\det\!\big(\varphi_{1,\sigma}-1_{n-1},
 \varphi_{2,\sigma}-1_{n-1},\dots,\varphi_{n-1,\sigma}-1_{n-1}\big).
 \end{equation}
 Since $ \varphi_{i,\sigma}=\ell(f_{i+1,\sigma})$,
 the above $(n-1)\times(n-1)$ determinant
  is related to the $n\times n$ determinant in the
 lemma by the identity
 \begin{align}\nonumber
\mathrm{sign}\Big(\!&\det\!\big(1_n,w_2,\dots,w_n\big)\Big)=\\ &
(-1)^{n-1}\mathrm{sign}\Big(\!\det\!\big(\ell(w_2)-1_{n-1},
\ell(w_3)-1_{n-1},\dots,\ell(w_n)-1_{n-1}\big)\Big),\label{trans}
\end{align}
valid for any $w_i\in\R^n$ with $w_i^{(n)}>0\
\,(2\le i\le n)$.\footnote{\ To prove \eqref{trans},
start with the matrix $(1_n,w_2,\dots,w_n)$, divide
 the $i^{\mathrm{th}}$ column (\ie $w_i $) by $w_i^{(n)}$ for
 $2\le i\le n$. This makes no change in the sign of
  the determinant as $w_i^{(n)}>0$.
  Now subtract the first column $1_n$
  from each of the other columns and expand by the last row.}

  Combining \eqref{degAsigma}, \eqref{detLsigma} and
   \eqref{trans} gives  the lemma.
\end{proof}

We now prove the basic count \eqref{AIM} at a  generic point, \ie for  $y\in\R^{n-1}_+-\mathcal{B}$, where
\begin{equation}\label{B}
\mathcal{B}:=\bigcup_{\sigma\in S_{n-1}} \mathcal{B}_\sigma,\qquad\qquad \mathcal{B}_\sigma:=
\bigcup_{\widetilde{\varepsilon}\in\widetilde{V}}\widetilde{\varepsilon}
\cdot\partial \overline{c}_\sigma.
\end{equation}
 Note that $\overline{c}_\sigma\subset\mathcal{B}$  when $w_\sigma=0$, for then
 $\overline{c}_\sigma$ coincides with its boundary $\partial \overline{c}_\sigma$.

 Let $\alpha:=\pi(y)\in T-\pi(\mathcal{B})$. By the remark immediately following Lemma \ref{degF},
 $F^{-1}(\alpha)\not=\varnothing$. Let
  $\delta\in F^{-1}(\alpha)\subset\widehat{T}$,
  and suppose $x\in [0,1]^{n-1}$ satisfies $\widehat{\pi}(x)=\delta$.
 Then
 $\alpha=F\big(\widehat{\pi}(x)\big)=\pi\big(f(x)\big)$. If
  we had $x\in\partial D_\sigma$ for some $\sigma\in S_{n-1}$,  then $f(x)\in
  f(\partial D_\sigma)\subset\partial\overline{c}_\sigma\subset \mathcal{B}$,
  contradicting $\alpha\notin\pi(\mathcal{B})$. Thus, $x\notin \partial D_\sigma$ for
  any $\sigma\in S_{n-1}$. Similarly, $x\notin D_\sigma$ for any $\sigma\in S_{n-1}$
  such that $w_\sigma=0$.  If  $w_\sigma\not=0$, the map $f=A_\sigma$  (see Proposition \ref{f})
  gives a bijection between the interior of $D_\sigma$ and the interior of $\overline{c}_\sigma$.
   It follows that  $f$ is a local homeomorphism in a neighborhood of $x$, as are
  $\widehat{\pi}$ and $\pi$ $\big($the latter in a neighborhood of $f(x)\big)$. Hence $F$ is a local homeomorphism
  in a neighborhood of $\delta$. Thus, $\delta=\widehat{\pi}(x)$  with $x$ in the interior of some $D_\sigma$, and
  $w_\sigma\not=0$.   Moreover, as $\widehat{\pi}$  restricted to $(0,1)^{n-1}$ is a bijection onto its image,
   there is a   unique point   $x\in\widehat{\pi}^{-1}(\delta)$. Also, $f(x)$ is in the interior of $\overline{c}_\sigma$,
   so $f(x)\in c_\sigma$.

   We now calculate using
   Lemma \ref{goodcasedegree},
  Proposition  \ref{DegreeProps}$\,$\eqref{CompactCase} and \eqref{Localglobal},
\begin{align*} 
 \deg(F) & = \deg_{\alpha}(F)=\sum_{\delta\in F^{-1}(\alpha)}\mathrm{locdeg}_\delta (F)=
 \sum_{\substack{ \sigma\in S_{n-1}\\ w_\sigma\not=0}}
        \  \sum_{\substack{ x\in D_\sigma \\ \widehat{\pi}(x)\in F^{-1}(\alpha)}}\!  \mathrm{locdeg}_{\widehat{\pi}(x)} (F)
           \\ &
    = \sum_{\substack{ \sigma\in S_{n-1}\\ w_\sigma\not=0}}
       \   \sum_{\substack{ x\in D_\sigma \\ F(\widehat{\pi}(x))=\alpha}}\!  v_\sigma
    = \sum_{\substack{ \sigma\in S_{n-1}\\ w_\sigma\not=0}}
        \  \sum_{\substack{ x\in D_\sigma \\ \pi(f(x))=\pi(y)}} \! v_\sigma
         = \sum_{\substack{ \sigma\in S_{n-1}\\ w_\sigma\not=0}}\
        \  \sum_{\substack{ x\in D_\sigma \\  f(x)\in\widetilde{V}\cdot y}}\!  v_\sigma
    \\ &            =
          \sum_{\substack{ \sigma\in S_{n-1}\\ w_\sigma\not=0}} \
          \sum_{  z\in c_\sigma\cap\widetilde{V}\cdot y }\!  v_\sigma   = \deg(F)
          \sum_{\substack{ \sigma\in S_{n-1}\\ w_\sigma\not=0}}\
          \sum_{  z\in c_\sigma\cap\widetilde{V}\cdot y }\!  w_\sigma ,
\end{align*}
since  $v_\sigma=\deg(F) w_\sigma $ by
 \eqref{wsigma},\,\eqref{degreeFlemma}    and  \eqref{vsigma}.
The main count \eqref{AIM}, for $y\in\R^{n-1}_+-\mathcal{B}$, follows   on dividing both sides   by $\deg(F)=\pm1$.

\subsection{End of proof of Theorem \ref{Main}}\label{BOUNDARY}
We now   address  $\widetilde{V}$-orbits which may intersect
 the  boundary $\partial \overline{c}_\sigma$ of some $\overline{c}_\sigma\subset\R^{n-1}_+$.
 For $y\in\R^{n-1}_+$ and $\sigma\in S_{n-1}$,
define  $J_\sigma(y)\subset\widetilde{V}$ as
  \begin{equation}\label{Jsigma}
J_\sigma(y):=\big\{\widetilde{\varepsilon}\in\widetilde{V}\big|\,\widetilde{\varepsilon}\cdot y\in c_\sigma\big\}.
 \end{equation}
As noted at the beginning of \S\ref{finalsection}, $J_\sigma(y)$
is a finite (possibly empty) set for any $y\in \R^{n-1}_+$. The point of defining $J_\sigma$ is that
 \begin{equation}\label{Jsigma2}
 \sum_{ z\in c_\sigma\cap\widetilde{V}\cdot y  }\!   w_\sigma
 = w_\sigma\, \mathrm{Card}\big(J_\sigma(y)\big)\qquad\qquad(y\in\R^{n-1}_+ ,\ \,\sigma\in S_{n-1}).
 \end{equation}
 Recall that we defined $\mathcal{B}_\sigma $ in \eqref{B} as the $\widetilde{V}$-orbit of the boundary of
$\overline{c}_\sigma$.
\begin{lemma}\label{Jpoints} For $y\in\R^{n-1}_+ $ and $\sigma\in S_{n-1}$,
there exists $T_\sigma(y)\in(0,1)$ such that $T_\sigma(y)\le t<1$
implies $J_\sigma(y)=J_\sigma(ty) $ and  $ty\notin\mathcal{B}_\sigma$.
\end{lemma}
\begin{proof} We first deal with the $J_\sigma$'s. Suppose $\widetilde{\varepsilon} \in J_\sigma(y)$,
so $\widetilde{\varepsilon}\cdot y\in c_\sigma$.  Lemma \ref{csigmareason} shows that $\overrightarrow{0,
\widetilde{\varepsilon}\cdot y} $ pierces $\overline{c}_\sigma$. By Definition \ref{defpiercing}, this means that
there is some $z=t_{\widetilde{\varepsilon}} (\widetilde{\varepsilon}\cdot y)$,
with $0\le t_{\widetilde{\varepsilon}}\le 1$, such that $z\in \
\stackrel{\circ}{\overline{c}}_\sigma$, \ie $z$ lies in the interior of
$\overline{c}_\sigma$. We cannot have
 $t_{\widetilde{\varepsilon}}=0$ as $\overline{c}_\sigma\subset\R^{n-1}_+$ lies
  in the strictly positive orthant. If $t_{\widetilde{\varepsilon}}=1$, then
  $\widetilde{\varepsilon}\cdot y$ itself is interior to $\overline{c}_\sigma$, so we
may reduce $t_{\widetilde{\varepsilon}}$ so that $0<t_{\widetilde{\varepsilon}}<1$. As
 $z=t_{\widetilde{\varepsilon}} (\widetilde{\varepsilon}\cdot y)
 \in\ \stackrel{\circ}{\overline{c}}_\sigma$,
 the last claim in Lemma \ref{piercing} shows that
  $t( \widetilde{\varepsilon}\cdot y)\in \
  \stackrel{\circ}{\overline{c}}_\sigma\subset c_\sigma$ for
     $t_{\widetilde{\varepsilon}}\le t<1$. As $t(\widetilde{\varepsilon}\cdot y)
= \widetilde{\varepsilon}\cdot ty$, we have shown $J_\sigma(y)\subset J_\sigma(ty)$ for $T_0\le t<1$,
where $T_0 :=\max_{\widetilde{\varepsilon}\in J_\sigma(y)}\big\{t_{\widetilde{\varepsilon}} \}<1$.

We now prove that   $J_\sigma(ty)\subset  J_\sigma(y)$ for all $t<1$ sufficiently close to 1.
Assume this is  false. Then there is a sequence $\{t_j\}_j$, with $0<t_j<1$, converging  to 1
with $J_\sigma(t_j y)\not\subset  J_\sigma(y)$, \ie for each $j$ there
is some $\widetilde{\varepsilon}_j\in \widetilde{V}$ such that $\widetilde{\varepsilon}_j\cdot t_jy\in c_\sigma$, but
$\widetilde{\varepsilon}_j\cdot y\notin c_\sigma$. Since all but a
finite number of $\widetilde{\varepsilon}\in \widetilde{V}$
take a small neighborhood of $y$ to the complement of $\overline{c}_\sigma$,
 the  $\widetilde{\varepsilon}_j$ range over a finite subset of $ \widetilde{V}$. By passing to a
 subsequence of the $t_j$'s (which we again denote by $t_j$),
  we can assume  $\widetilde{\varepsilon}_j =  \widetilde{\varepsilon}$,
 a fixed element of   $\widetilde{V}$ with
$\widetilde{\varepsilon}\notin J_\sigma(y)$. By  Lemma \ref{csigmareason},
  $\overrightarrow{0,\widetilde{\varepsilon}\cdot t_jy} $ pierces $\overline{c}_\sigma$.
In particular, $\widetilde{\varepsilon}\cdot t_j y\in\overline{c}_\sigma$. Since  $\overline{c}_\sigma$
is closed and $t_j\to1$, we see that
 $\widetilde{\varepsilon}\cdot  y\in\overline{c}_\sigma$. But
  $\overrightarrow{0,\widetilde{\varepsilon}\cdot t_jy} $ intersects $\stackrel{\circ}{\overline{c}}_\sigma$, as
 it pierces $\overline{c}_\sigma$. Now,  $\overrightarrow{0,\widetilde{\varepsilon}\cdot y} $ contains
 $\overrightarrow{0,\widetilde{\varepsilon}\cdot t_jy} $, so it also  pierces
 $\overline{c}_\sigma$. But   Lemma  \ref{csigmareason} implies that $\widetilde{\varepsilon}\cdot y\in c_\sigma$,
 contradicting $\widetilde{\varepsilon}\notin J_\sigma(y)$. Hence $J_\sigma(y)=J_\sigma(ty)$ for all $t<1$ near enough
 to 1, as claimed.

 We now prove the last claim in the lemma, namely that $ty\notin \mathcal{B}_\sigma$ for all $t$
  sufficiently close to 1.
 If this is false,  there is again a sequence
 $\{t_j\}_j$, with $0<t_j<1$, converging to 1 such that $t_jy\in \mathcal{B}_\sigma$, \ie  for each $j$ there
is some $\widetilde{\varepsilon}_j\in \widetilde{V}$ such that
$\widetilde{\varepsilon}_j\cdot t_j y\in \partial \overline{c}_\sigma$.
Passing to a subsequence, we can  assume that
$\widetilde{\varepsilon}\cdot t_j y\in \partial \overline{c}_\sigma$ for
 some $\widetilde{\varepsilon}\in \widetilde{V}$. But
the boundary $ \partial \overline{c}_\sigma$ lies in the union of the affine subspaces
 $h_{i,\sigma}$ $\big($see \eqref{defhisigma}$\big)$ extending the faces
of $\overline{c}_\sigma\ \,(0\le i\le n-1)$. Passing again to a subsequence, we can assume that
$\widetilde{\varepsilon}\cdot t_j y\in   h_{i_0,\sigma}$, for a fixed $i_0$. Since $h_{i_0,\sigma}$ is an affine
subspace, and it contains more than one point on the straight line connecting 0 and $\widetilde{\varepsilon}\cdot y$, it must contain the entire line. In
particular, $0\in  h_{i_0,\sigma}$, contradicting Lemma \ref{goodpoint}.
\end{proof}

We   now  conclude the proof of the main count \eqref{AIM} for any $y\in\R^{n-1}_+$. The above lemma
shows the existence of $y_0=y_0(y)\in \R^{n-1}_+$
 such that $J_\sigma(y_0) = J_\sigma(y)$ and $y_0\notin \mathcal{B}_\sigma$ for all $\sigma\in S_{n-1}$.
 Thus $y_0\notin\mathcal{B}:=\bigcup_\sigma \mathcal{B}_\sigma$. In particular, from the previous subsection,
 we know that \eqref{AIM} holds for $y_0$.
 Hence, using \eqref{Jsigma2},
\begin{align*}
 1   &=   \sum_{\substack{\sigma\in S_{n-1}\\ w_\sigma\not=0}}
             \ \sum_{z\in c_\sigma\cap\widetilde{V}\cdot y_0} w_\sigma =
             \sum_{\substack{\sigma\in S_{n-1}\\ w_\sigma\not=0}} w_\sigma\  \mathrm{Card}\big(J_\sigma(y_0)\big)
         \\ &    = \sum_{\substack{\sigma\in S_{n-1}\\ w_\sigma\not=0}} \ w_\sigma\, \mathrm{Card}\big(J_\sigma(y)\big)
         = \sum_{\substack{\sigma\in S_{n-1}\\ w_\sigma\not=0}}
             \,\sum_{z\in c_\sigma\cap\widetilde{V}\cdot y} w_\sigma. \qed
\end{align*}

  \end{document}